\documentclass[reqno,11pt]{amsart}

	\usepackage[margin=1.3in]{geometry}

	\usepackage[utf8]{inputenc}

	\usepackage{amsmath,amsfonts,amssymb,amsthm}

	\usepackage[T1]{fontenc}

	\usepackage{etoolbox}
	\patchcmd{\section}{\scshape}{\scshape\bfseries}{}{}
	\makeatletter
	\renewcommand{\@secnumfont}{\scshape\bfseries}
	\makeatother

	\let\epsilon\varepsilon

	\usepackage{parskip}

	\usepackage[bookmarks=true,bookmarksopen=true]{hyperref}

	\newtheorem{theorem}{Theorem}

	\newtheorem{corollary}[theorem]{Corollary}
	\newtheorem{lemma}[theorem]{Lemma}
	\newtheorem{proposition}[theorem]{Proposition}
	
	\numberwithin{theorem}{section}
	
	\theoremstyle{plain}
	\newtheorem{result}{Theorem}

	\theoremstyle{definition}
	\newtheorem*{definition}{Definition}

	\theoremstyle{remark}
	\newtheorem{remark}[theorem]{Remark}
	
	\numberwithin{equation}{section}
	
	\numberwithin{table}{section}

	\usepackage{tikz}
	\usetikzlibrary{matrix}

	\usepackage{todonotes}

	\usepackage{enumerate}

	\usepackage{float}

	\makeatletter
	\def\blfootnote{\gdef\@thefnmark{}\@footnotetext}
	\makeatother

	\DeclareMathOperator{\codim}{codim}

	\DeclareMathOperator{\diag}{diag}

	\DeclareMathOperator{\Ric}{Ric}

	\newcommand{\inj}{\hookrightarrow}

	\newcommand{\of}[1]{\left( #1 \right)}
	
	\newcommand{\floor}[1]{\left\lfloor #1 \right\rfloor}

	\newcommand{\pr}[1]{\Ric_{#1} > 0}

	\newcommand{\gG}{\mathsf{G}}

	\newcommand{\gH}{\mathsf{H}}

	\newcommand{\gO}{\mathsf{O}}
	
	\newcommand{\gS}{\mathsf{S}}
	
	\newcommand{\SO}{\mathsf{SO}}

	\newcommand{\Spin}{\mathsf{Spin}}
	\newcommand{\SU}{\mathsf{SU}}
	\newcommand{\gT}{\mathsf{T}}
	\newcommand{\gU}{\mathsf{U}}
	
	\newcommand{\RP}{\mathbb{R}\mathrm{P}}
	\newcommand{\CP}{\mathbb{C}\mathrm{P}}
	\newcommand{\mCP}[1]{\overline{\mathbb{C}\mathrm{P}^{#1}}}
	\newcommand{\HP}{\mathbb{H}\mathrm{P}}
	\newcommand{\OP}{\mathbb{O}\mathrm{P}^2}

	\newcommand{\bR}{\mathbb{R}}
	
	\newcommand{\bZ}{\mathbb{Z}}
	
	\newcommand{\ep}{\epsilon}
	
	\mathcode`l="8000
	\begingroup
	\makeatletter
	\lccode`\~=`\l
	\DeclareMathSymbol{\lsb@l}{\mathalpha}{letters}{`l}
	\lowercase{\gdef~{\ifnum\the\mathgroup=\m@ne \ell \else \lsb@l \fi}}%
	\endgroup

	\newcommand*{\defeq}{\mathrel{\vcenter{\baselineskip0.5ex \lineskiplimit0pt
				\hbox{\scriptsize.}\hbox{\scriptsize.}}}%
		=}

	\usepackage{color}

	{\providecommand{\noopsort}[1]{} %year = "unpublished manuscript\setbox0=\hbox{2003}"

\title[$\pr{k}$ with maximal symmetry rank]{Positive intermediate Ricci curvature with maximal symmetry rank}
\date{\today}

\author{Lee Kennard}
\address{Department of Mathematics \\ Syracuse University \\ Syracuse, NY USA.}
\email{ltkennar@syr.edu}

\author{Lawrence Mouill\'e}
\address{Department of Mathematics \\ Syracuse University \\ Syracuse, NY USA.}
\email{lawrence.mouille@gmail.com}

\begin{document}

\begin{abstract}
	Generalizing the foundational work of Grove and Searle, the second author proved upper bounds on the ranks of isometry groups of closed Riemannian manifolds with positive intermediate Ricci curvature and established some topological rigidity results in the case of maximal symmetry rank and positive second intermediate Ricci curvature. 
	Here, we recover even stronger topological rigidity, including results for higher intermediate Ricci curvatures and for manifolds with nontrivial fundamental groups. 
\end{abstract}

\maketitle

\blfootnote{\textup{2020} \textit{Mathematics Subject Classification}: 53C20 (Primary), 57S15 (Secondary)}

\section{Introduction}
	
	The study of manifolds with lower curvature bounds goes back to the origins of Riemannian geometry.
	In dimensions at least $3$, there are many notions of lower curvature bounds, two of the most common being positive sectional curvature and positive Ricci curvature. 
	For sectional curvature, the classification of positively curved manifolds is a wide open problem.
	However, apart from the compact rank one symmetric spaces $S^n$, $\CP^n$, $\HP^n$, $\OP$, the only dimensions that are known to admit other closed, simply connected manifolds with positive sectional curvature are $6$, $7$, $12$, $13$, and $24$; see \cite{Ziller06}.
	In contrast, for the weaker condition of positive Ricci curvature, there are far more known examples, including sequences of examples in any fixed dimension $\geq 4$ with unbounded total Betti numbers \cite{ShaYang91}; see also \cite{Wraith07,Burdick19,Burdick20,Reiser23,Reiser22preprint}.
	
	For positive sectional curvature, the Grove Symmetry Program has resulted in major classifications such as those in  \cite{FangGroveThorbergsson17, FangRong05, GorodskiKollrossWilking21preprint, GroveSearle94, GroveSearle97, GroveWilking14, GroveWilkingZiller08, HsiangKleiner89, Rong02, Wilking03, Wilking06, VerdianiZiller18}, constructions of new examples of manifolds with lower curvature bounds as in \cite{Dearricott11, GoetteKerinShankar20, GoetteKerinShankar21, GroveVerdianiZiller11, PetersenWilhelm08preprint}, and discoveries of unexpected and fundamental connections between curvature and topology \cite{Shankar98, Wilking07}.
	The overarching goal of this program, initiated by Karsten Grove in the 1990s, is to classify positively curved spaces with large isometry groups.
	A foundational result in the program is Grove and Searle's maximal symmetry rank theorem: 
	Any closed, connected, $n$-dimensional manifold with positive sectional curvature has symmetry rank (i.e. rank of the isometry group) bounded above by $\lfloor \frac{n+1}{2} \rfloor$, and in the case of equality (i.e. maximal symmetry rank), the manifold is diffeomorphic to $S^n$, $\bR\mathrm{P}^n$, $\CP^{n/2}$, or a lens space \cite{GroveSearle94}.
	Galaz-Garc\'ia later strengthened this conclusion to an equivariant diffeomorphism classification \cite{GalazGarcia14}.
	
	Because of the success of the Grove Symmetry Program, it is natural to ask which obstructions to positive sectional curvature generalize to weaker curvature conditions (e.g. positive Ricci curvature, non-negative sectional curvature, quasipositive curvature, almost positive curvature).
	However, it is typical that tools from the positive sectional curvature setting do not carry over to weaker curvature conditions.
	One exception was established by the second author for positive intermediate Ricci curvature.
	
	\begin{definition}
		Given an $n$-dimensional Riemannian manifold $M$ and $k \in \{ 1 , \dots , n - 1 \}$, we say $M$ has \textit{positive $k^{th}$-intermediate Ricci curvature ($Ric_k > 0$)} if, for every set of orthonormal vectors $x,y_1,\dots,y_k$ tangent to $M$, the sum of sectional curvatures $\sum_i \sec(x,y_i)$ is positive.\footnote{This notion of positive intermediate Ricci curvature should not be confused with $k$-positive Ricci curvature as defined in \cite{Wolfson11}; see also \cite{CrowleyWraith20preprint,WalshWraith22}.}
	\end{definition}
	
	Note $M$ has $\pr{k}$ if and only if, for every unit vector $x$ tangent to $M$, the sum of any $k+1$ eigenvalues of the Jacobi (directional curvature) operator $y \mapsto R(y,x)x$ is positive; see \cite[Lemma 1.2]{ReiserWraith22preprint2}.
	Thus $\pr{1}$ is equivalent to positive sectional curvature, $\pr{n-1}$ is equivalent to positive Ricci curvature, and if $\pr{k}$ for some $k$, then $\pr{l}$ for all $l\geq k$.
	Because the condition $\pr{k}$ is vacuous for dimensions $n \leq k$, we use the convention that an assumption of $\pr{k}$ implies $n \geq k+1$.
	
	Several results for manifolds with sectional curvature lower bounds have been extended to intermediate Ricci curvature lower bounds.
	These include generalizations of the Synge theorem and Weinstein fixed point theorem \cite{Wilhelm97}, the Gromoll-Meyer theorem and Cheeger-Gromoll Soul theorem \cite{Shen93,GuijarroWilhelm20}, the quarter-pinched sphere theorem \cite{Shen90,Xia97,GuijarroWilhelm22}, and the Heintze-Karcher inequality \cite{Chahine19}.
	Also, comparison results have been established by Guijarro and Wilhelm \cite{GuijarroWilhelm18,GuijarroWilhelm22,GuijarroWilhelm20}, and examples and constructions can be found in \cite{AmannQuastZarei20preprint,DVGAM22,ReiserWraith23,ReiserWraith23preprint1,ReiserWraith23preprint2,ReiserWraith22preprint2}.
	For a collection of publications and preprints concerning intermediate Ricci curvature, see \cite{MouilleWebpage}.
	
	In \cite{Mouille22a,Mouille22b}, the second author shows that closed, connected, $n$-dimensional manifolds with $\pr{2}$ have symmetry rank bounded above by $\lfloor \frac{n+1}{2} \rfloor$, the same bound Grove and Searle established for positive sectional curvature. 
%	Furthermore, the author showed in \cite{Mouille22b} that odd-dimensional, closed, simply connected manifolds with $\pr{2}$ and maximal symmetry rank are diffeomorphic to spheres.
{
Furthermore, the second author proved the following rigidity statement in the odd-dimensional case:% (see ):

\begin{theorem}[\cite{Mouille22b}]\label{thm:odd-dim}
If $\gT^n$ acts effectively by isometries on a closed, simply connected Riemannian $(2n-1)$-manifold $M$ with $\pr2$, then $M$ is diffeomorphic to $S^{2n-1}$.
\end{theorem}

The second author also proved a rigidity statement in even dimensions under the additional assumptions that the dimension is at least eight and the second Betti number is at most one. Here, we remove both of these assumptions and prove the following, our main result:
}

%	In this article, we establish a classification of even-dimensional closed, simply connected manifolds with $\pr{2}$ and maximal symmetry rank:
	
	\begin{result}\label{result:k=2}
		If $\gT^n$ acts effectively by isometries on a closed, 
		simply connected Riemannian $(2n)$-manifold $M$ with $\pr2$,
%		Let $M$ be a closed, simply connected Riemannian manifold of dimension $2n$ with $\pr{2}$.
%		If $\gT^n$ acts effectively by isometries on $M$,
		then one of the following holds:
		\begin{enumerate}
			\item $M$ is at least six-dimensional and is diffeomorphic to $S^{2n}$ or homeomorphic to $\CP^n$,\label{item:mainS2nCPn}
			\item $M$ is six-dimensional and $\chi(M) = 0$, or \label{item:maindim6}
			\item $M$ is equivariantly diffeomorphic to {$S^4$} with a linear $\gT^2$-action or to an equivariant connected sum $\CP^2 \# \dots \# \CP^2$ with a linear $\gT^2$-action on each summand. \label{item:mainconnsum}
		\end{enumerate}
	\end{result}
	
	Since $S^{2n}$ and $\mathbb C \mathrm P^n$ admit metrics with positive sectional curvature and $\gT^n$ symmetry, {Theorems \ref{thm:odd-dim} and \ref{result:k=2} } may be viewed as providing a complete homeomorphism classification in dimensions at least seven of manifolds admitting metrics with $\Ric_2 > 0$ and maximal symmetry rank. 
	As our methods differ from those in Grove-Searle \cite{GroveSearle94}, we do not know whether it is possible to upgrade the rigidity to equivariant diffeomorphism as in the case of positive sectional curvature, {even in view of the work of Montgomery and Yang that implies equivariance for a circle subaction but perhaps not for the torus action itself (see \cite{MontgomeryYang67} and Theorem \ref{thm:homotopysphere} below).}

	In dimension six, there is a metric on $S^3 \times S^3$ with $\pr{2}$ and maximal symmetry rank; see \cite[Example 2.3]{Mouille22b}. 
	We remark that, if additionally $M^6$ is $2$-connected, then the conclusion $\chi(M^6) = 0$ in (\ref*{item:maindim6}) is sufficient to imply that $M^6$ is diffeomorphic to $S^3 \times S^3$ (see Remark \ref{rem:2-connected} below). 
	We do not know if other closed, simply connected $6$-manifolds with $\chi(M) = 0$ admit metrics with $\Ric_2 > 0$.
	We remark that if a $\gT^3$-action $M^6$ has a fixed point, then we prove in Proposition \ref{prop:dim6fixedpt} below that $M^6$ is diffeomorphic to either $S^6$ or $\CP^3$.

	In dimension four, Orlik and Raymond showed that a smooth, simply connected, closed four-manifold with $\gT^2$ symmetry is equivariantly diffeomorphic to a connected sum of {$S^4$ and } copies of $S^2 \times S^2$, $\CP^2$, and $\mCP{2}$ (i.e. $\CP^2$ with the opposite orientation) \cite{OrlikRaymond70}. 
	Theorem \ref{result:k=2} shows that if additionally $M^4$ admits a $\gT^2$-invariant metric with $\pr{2}$, then $S^2 \times S^2$ summands do not appear and moreover that all summands of $\mathbb C \mathrm P^2$ come with the same orientation. 
	However, we currently cannot rule out the cases $b \geq 2$, and for these values, it is unknown whether any such manifold admits $\pr{2}$, much less whether such a metric can be invariant under a $\gT^2$-action. 
	Finally, we note that though $S^2 \times S^2$ cannot admit a metric with $\pr{2}$ and $\gT^2$-symmetry, it does admit one with $\gS^1$-symmetry; see \cite[Example 2.3]{Mouille22b}.
	In fact, this $\pr{2}$ metric on $S^2 \times S^2$ is invariant under a cohomogeneity one action by $\SO(3)$.
	
	We remark now on other known generalizations of Grove and Searle's maximal symmetry rank theorem. 
	First, in an unpublished manuscript, Wilking extended the Grove and Searle symmetry rank bound to manifolds that contain a point at which all sectional curvatures are positive \cite{Wilking08preprint}; for the proof, see \cite[Theorem 1.3]{GalazGarcia14}.
	Galaz-Garc\'ia then extended the Grove and Searle classification for maximal symmetry rank to manifolds with quasipositive curvature (sectional curvature non-negative everywhere and positive at a point) in dimensions four and five \cite{GalazGarcia14}.
	The second author established a generalized version of Wilking's symmetry rank bound for manifolds which have a point at which all intermediate Ricci curvatures are positive \cite{Mouille22a}.
	
	Second, for the case of positive weighted sectional curvature, the first author and Wylie proved the symmetry rank bound is the same as for positive sectional curvature, and they recover rigidity in the equality case up to homeomorphism; see \cite{KennardWylie17}.
	
	Third, for non-negative sectional curvature, Galaz-Garc\'ia and Searle conjectured a generalization of the maximal symmetry rank theorem \cite{GalazGarciaSearle11}, which was later reformulated and sharpened by Escher and Searle \cite{EscherSearle21}.
	%The present conjecture is that simply connected, closed $n$-dimensional manifolds with non-negative sectional curvature must have symmetry rank bounded above by $\lfloor \frac{2n}{3} \rfloor$, and in the case of equality, must be equivariantly diffeomorphic to {a linear action on } a product of spheres or a quotient thereof by a {free } linear torus action.
	{Work of Galaz-Garc\'ia and Searle \cite{GalazGarciaSearle11}, Galaz-Garc\'ia and Kerin \cite{Galaz-GarciaKerin14}, and Escher and Searle \cite{EscherSearle21} confirm this conjecture up to dimension nine and moreover prove the symmetry rank upper bound in dimensions up to $12$. } 
	% proved this conjecture in dimensions up to $6$, and and Escher and Searle showed that the conjectured symmetry rank bound holds up to dimension $12$, and the classification up to dimension $9$ \cite{EscherSearle21}.
	With the added assumption that the maximal torus action is isotropy-maximal, the conjecture was shown by Escher and Searle to hold in all dimensions in \cite{EscherSearle21}.
	{
	Analogously, the conjecture was established up to rational homotopy equivalence by Galaz-Garc\'ia, Kerin, and Radeschi in \cite{GGKR20} in the case where the assumption of non-negative sectional curvature is replaced with rational ellipticity, which is expected to follow from non-negative sectional curvature by the Bott-Grove-Halperin ellipticity conjecture (see \cite{GroveHalperin82,Grove02,GroveWilkingYeager19}).%, then the conjecture was established up to rational homotopy equivalence by Galaz-Garc\'ia, Kerin, and Radeschi in \cite{GGKR20}.
	}
	
	Finally, for the situation where the torus is replaced by an elementary $p$-group for some prime $p$, Fang and Rong proved the optimal upper bound and obtained homeomorphism rigidity in the equality case for $p$ larger than a constant depending only on the manifold dimension \cite{FangRong04}.
	There are two analogues of this result for $p = 2$ (see \cite{FangGrove16} and \cite[Theorems A and B]{KennardSamaniSearle21preprint}).
	
	Our second main result is a rigidity statement for Riemannian manifolds with $\pr{k}$ for larger values of $k$. 
	More precisely, given a closed connected $n$-dimensional Riemannian manifold $M$, the second author proved that the symmetry rank of $M^n$ is at most $\floor{\tfrac{n+k}{2}} - 1$ if $M^n$ has $\pr k$ with $k \geq 3$ (see \cite[Proposition 1.6]{Mouille22b}). This bound agrees with the classical bound of $\floor{\frac{n+1}{2}}$ for $k = 3$ and for $k = 4$ when $n$ is odd. 
	Since the condition $\pr{k}$ grows weaker as $k$ increases, the available tools also grow weaker for manifolds with $\pr{k}$, and one should not expect to be able to prove an analogue of Theorem \ref{result:k=2} in this setting without stronger hypotheses. We prove two results along these lines. The first is based on the model spaces of spheres and $S^3 \times S^3$.

	\begin{result}\label{result:k>2sphere}
	Fix $k \geq 3$, and assume $M^n$ is a $(k-1)$-connected, closed Riemannian manifold with $n\neq 7$ if $k = 3$. If $M^n$ has $\Ric_k > 0$ and admits an isometric $\gT^r$-action with $r = \floor{\frac{n+k}{2}} - 1$, then one of the following occurs:
		\begin{enumerate}
		\item $M$ is diffeomorphic to $S^n$ and $k \leq 4$, with equality only if $n$ is odd.
		\item $M$ is diffeomorphic to $S^3 \times S^3$ and $k = 3$.
		\end{enumerate}
	\end{result}

The conclusions in (1) and (2) are optimal in the sense that $S^n$ and $S^3 \times S^3$ admit metrics with $\Ric_k > 0$ and maximal symmetry rank for all values of $k$ shown. The second result for large values of $k$ is modeled on complex projective space:

	\begin{result}\label{result:k>2CP}
		Fix $k \geq 3$, and let $M^n$ be a simply connected, closed Riemannian manifold.
		Assume further that $M^n$ is an integral cohomology $\CP$ up to degree $k + 2$. 
		If $M$ has $\Ric_k > 0$ and admits an effective, isometric $\gT^r$-action with $r = \floor{\frac{n+k}{2}} - 1$, then $n$ is even, $M^n$ is homeomorphic to $\CP^{\frac n 2}$, and $k = 3$.
	\end{result}
	
	By the assumption on the cohomology of $M$, we mean that $H^1(M;\bZ) \cong 0$, $H^2(M;\bZ)\cong \bZ$, and the map $H^i(M;\bZ) \to H^{i+2}(M;\bZ)$ induced by multiplication by a generator $x\in H^2(M;\bZ)$ is surjective for $0 \leq i < k$ and injective for $0 < i \leq k$.
	As with Theorem \ref{result:k=2}, we cannot obtain rigidity up to diffeomorphism, and we do not know whether any exotic $\CP^{\frac n 2}$ admits $\pr{3}$ and maximal symmetry rank.
	
	Though Theorems \ref{result:k>2sphere} and \ref{result:k>2CP} partially generalize Theorem \ref{result:k=2} under stronger topological assumptions, we note that there are potentially more examples of manifolds that satisfy $\pr{k}$ with $k \geq 3$ than those listed in the conclusions of these results. For example, under the respective product metrics, $S^2 \times S^2$ has $\pr{3}$ with $\gT^2$-symmetry, $S^3 \times S^2$ has $\pr{4}$ with $\gT^3$-symmetry, and $S^3 \times S^3$ has $\pr{4}$ and $\gT^4$-symmetry.
	
	Finally, we analyze the case of non-trivial fundamental group:
	
	\begin{result}\label{result:pi1}
	Let $M^n$ be a closed, connected Riemannian manifold with $\pr{2}$ and $\gT^r$ symmetry with $r = \left\lfloor \frac{n+1}{2} \right\rfloor$. If $\pi_1(M)$ is non-trivial, then one of the following occurs:
		\begin{enumerate}
		\item $M$ is homotopy equivalent to $\RP^n$ or a lens space, or
		\item $M$ has dimension six and, if additionally the universal cover is $S^3 \times S^3$, then $\pi_1(M) \cong \bZ_l \times \bZ_m$ for some $l,m\geq 1$.
		\end{enumerate}% or $n$ is odd and $M$ is homotopy equivalent to a lens space $S^n/\bZ_l$ with $l \geq 3$ and $\bZ_l$-action given by $\omega \cdot (z_0,\ldots,z_m) = (\omega^d
	\end{result}
	
	The standard models in (1) can already be realized with positive sectional curvature, and in the case of maximal symmetry rank and positive sectional curvature, Grove and Searle's result recovers rigidity up to diffeomorphism instead of just homotopy, and Galaz-Garc\'ia later strengthened this rigidity to equivariant diffeomorphism. 
	Also note that our result rules out the possibility that the universal cover is homeomorphic to $\CP^{\frac n 2}$ when $M$ is not simply connected. 
	Finally, we remark that the fundamental groups as in (2) can be realized by the known metric on $S^3 \times S^3$ with $\gT^3$ symmetry and $\Ric_2 > 0$. 
	Indeed, the product of the Hopf actions on the $S^3$ factors gives rise to a $\gT^2$-subaction of the $\gT^3$-action. Since this action is free, we obtain products of lens space $S^3/\bZ_l \times S^3/\bZ_m$ as examples, as well as quotients $(S^3 \times S^3)/\bZ_l$ by possibly diagonal actions, for example, $(S^3 \times S^3)/\{\pm(1,1)\} = \SO(4)$. 
	We also remark that other finite groups act freely on $S^3 \times S^3$, including all finite subgroups of $\Spin(4) \cong S^3 \times S^3$ and, more surprisingly, the two-fold product $S_3 \times S_3$ of the symmetric group on three letters (see Davis \cite{Davis98preprint} and \cite{Hambleton06}). 
	Finally, we note that Dom\'inguez-V\'azquez, Gonz\'alez-\'Alvaro, and Rodr\'iguez-V\'azquez have determined that the Wallach flag manifold $W^6 = \SU(3)/\gT^2$ with the normal homogeneous metric has $\pr{2}$ \cite{DVGARVinprep}. 
	This metric  has a free, isometric $S_3$-action and has $\gT^2$-symmetry, which is not maximal for $\pr{2}$ in dimension $6$.
	
	Regarding the non-simply connected case of Theorems \ref{result:k>2sphere} and \ref{result:k>2CP}, the conclusions are the same as in Part (1) of Theorem \ref{result:pi1} under the modification that the cohomology of the {\it universal cover} of $M$ satisfies the topological assumptions stated in Theorems \ref{result:k>2sphere} and \ref{result:k>2CP}. The proof is a straightforward modification of the proof we give for the case of $k = 2$, so it is omitted. 
	
% Proof of Theorem A
	The key tool for establishing Theorem \ref{result:k=2} is the Connectedness Lemma (Theorem \ref{thm:connprinc}).
	In dimension $4$, we use this tool in conjunction with topological results by Orlik and Raymond \cite{OrlikRaymond70}; see Section \ref{sec:dim4}.
	In dimension $6$, we break the argument into the cases according to whether the torus has a fixed point; see Section \ref{sec:dim6}.
	If the torus has a fixed point, we employ an argument involving Euler characteristics to rule out connected sums of complex projective spaces.
	Curiously, this argument only eliminates such connected sums in dimensions strictly larger than four.
	If the torus has no fixed points, then it follows immediately that $\chi(M) = \chi(M^{\gT^3}) = 0$. 
	For dimensions $8$ or greater, we prove that the torus has a fixed point, and the result then follows by induction using the Connectedness Lemma, noting in dimension eight that the induced torus action on the $6$-dimensional submanifold involved in the proof has a fixed point; see Section \ref{sec:dim8+}.

% Proof of Theorem B	
	To prove Theorems \ref{result:k>2sphere} and \ref{result:k>2CP}, we show in most cases that such manifolds must have a circle action whose fixed-point set contains a component of codimension $2$, and then we apply the Connectedness Lemma (Theorem \ref{thm:connprinc}); see Section \ref{sec:k>2}. 
	The demand for the additional topological assumptions 
	is a consequence of the fact that the Connectedness Lemma provides less information about the topology of $M$ as $k$ increases.

% Proof of Theorem C
	Theorem \ref{result:pi1} is proved in Section \ref{sec:pi1}. It borrows standard and elementary results from group cohomology that have been used previously in the positive sectional curvature case together with special arguments in the cases where the universal cover $\tilde M$ is diffeomorphic to $S^3 \times S^3$ or $\CP^2\#\ldots\#\CP^2$.

\subsection*{Acknowledgements}
	
	The authors would like to thank William Wylie, Jason DeVito, Marco Radeschi, and Fernando Galaz-Garc\'ia for helpful discussions, along with Jason DeVito and Philipp Reiser for suggestions on a previous version of this article. We also thank the referees for their comments to improve the exposition. 
	The first author was funded by NSF Grant DMS-2005280 and by Simons Foundation Award MP-TSM-00002791, and the second author was funded by NSF Award DMS-2202826.
	
\section{Preliminaries}
	
	We begin with a discussion of fixed-point sets.
	Given an isometric action of a Lie group $\gG$ on a Riemannian manifold $M$, we let $M^{\gG}$ denote the fixed-point set of the $\gG$-action on $M$.
	Given a point $p\in M^{\gG}$, we denote the component of $M^{\gG}$ that contains $p$ by $M_p^{\gG}$, and refer to it as a fixed-point component.
	The following is a foundational structure result for isometric torus actions.
	
	\begin{lemma}\label{lem:torusfpc}
		Let $M$ be a closed Riemannian manifold.
		Assume a torus $\gT^r$ acts isometrically on $M$, and let $\gH$ be a closed subgroup of $\gT^r$ whose fixed-point set $M^{\gH}$ is non-empty.
		Then every component of $M^{\gH}$ is an embedded, totally geodesic submanifold of even codimension in $M$ that is invariant under the action of $\gT^r / \gH$.
		Furthermore, given any fixed-point component $M_p^{\gH}$, the following hold:
		\begin{enumerate}
			\item If $\gH$ is a torus and $M$ is orientable, then $M_p^{\gH}$ is also orientable.
			
			\item If $\dim \gH \geq 2$, then there exists a circle subgroup $\gS^1 \subset \gH$ whose fixed-point component $M_p^{\gS^1}$ strictly contains $M_p^{\gH}$.\label{item:dimG>1}
			
			\item If $\gH$ is disconnected and is the isotropy group at $p$, and if $\dim \gH \geq 1$, then there exists a non-trivial, finite isotropy group $\Gamma \subseteq \gH$ whose fixed-point component $M_p^{\Gamma}$ strictly contains $M_p^{\gH}$. \label{item:Gdisconnected}
		\end{enumerate}
	\end{lemma}
	
	For justification for Part (\ref*{item:dimG>1}), see, for example, \cite[Proposition 8.3.8]{Petersen16}, and for Part (\ref*{item:Gdisconnected}), see \cite[Lemma 1.10]{KennardWiemelerWilking21preprint}.
	
	The next lemma, which was established by Conner \cite{Conner57}, will be especially useful in establishing Theorem \ref{result:pi1}.
	
	\begin{lemma}[Betti Number Lemma]\label{lem:Conner}
		If a torus $\gT$ acts smoothly on a closed, smooth manifold $M$, then $\chi(M) = \chi(M^{\gT})$, $\sum b_{2i+1}(M^{\gT}) \leq \sum b_{2i+1}(M)$, and $\sum b_{2i}(M^{\gT}) \leq \sum b_{2i}(M)$.
	\end{lemma}
	
	The two main ways in which our positive curvature assumptions play a role is via the following two results. 
	The first is a generalization of the Berger-Sugahara fixed-point theorem, which is stated in the next lemma.
	Part (\ref*{item:Berger}) was established by Berger \cite{Berger66}, the $k=1$ case of Part (\ref*{item:SugaharaMouille}) independently by Sugahara \cite{Sugahara82} and Grove and Searle \cite{GroveSearle94}, and the $k \geq 2$ case by the second author \cite{Mouille22b}.
	
	\begin{lemma}[Isotropy Rank Lemma]\label{lem:isotropyrank}
		Let $M$ be a closed Riemannian manifold with $\pr{k}$ and an isometric action by a torus $\gT^r$.
			\begin{enumerate}
			\item If $k = 1$ and $n$ is even, then $\gT^r$ has a fixed point. \label{item:Berger}
			\item For any $k\geq 1$ and $r \geq k$, there exists a subtorus $\gT^{r-k}$ that has a fixed point. \label{item:SugaharaMouille}
			\end{enumerate}
			%	If a torus $\gT^r$ of rank $r\geq k+1$ acts by isometries on $M$, then there exists a codimension-$k$ subtorus $\gT^{r-k} \subset \gT^r$ such that the $\gT^{r-k}$-action on $M$ has a fixed point. 
	\end{lemma}

{
We note that analogous conclusions hold if we replace the curvature assumption in Lemma \ref{lem:isotropyrank} by a topological one.
We believe this result is well known, but as we do not know of a reference, we provide a proof here for completeness. 
	\begin{lemma}[Spherical Isotropy Rank Lemma]\label{lem:isotropyrank_spheres}
		Let $M$ be a closed manifold with a smooth action by a torus $\gT^r$ with $r \geq 1$. If $M$ has the rational homology of a sphere, then the following hold:
			\begin{enumerate}
			\item If $n$ is even, then $\gT^r$ has a fixed point.
			\item If $n$ is odd, then some $\gT^{r-1}$ has a fixed point.
			\end{enumerate}
			%	If a torus $\gT^r$ of rank $r\geq k+1$ acts by isometries on $M$, then there exists a codimension-$k$ subtorus $\gT^{r-k} \subset \gT^r$ such that the $\gT^{r-k}$-action on $M$ has a fixed point. 
	\end{lemma}

\begin{proof}
Since $M$ and the fixed-point set of $M$ have the same Euler characteristic, which is non-zero for even-dimensional spheres, the first conclusion follows immediately.

We prove the second conclusion by induction over odd integers $n \geq 1$. When $n = 1$, $M$ is a circle so the kernel of the action on $M$ contains a $\gT^{r-1}$ and the result follows. Assume now that $n\geq 3$. The result holds trivially if $r = 1$, so we assume $r \geq 2$. Smith proved that $\mathbb{Z}_p \times \mathbb{Z}_p$ cannot act freely on a $\mathbb{Z}_p$-homology sphere (see \cite{Smith44}). It follows that $\gT^r$ cannot act almost freely on $M$, since otherwise we can find a prime $p$ sufficiently large so that $M$ is a $\mathbb{Z}_p$-homology sphere and so that the subgroup $\mathbb{Z}_p^r \subseteq \gT^r$ has trivial intersection with all isotropy groups and hence acts freely. We can now choose an isotropy group of positive dimension and hence a circle $\gS^1 \subseteq \gT^r$ with non-trivial fixed-point set $F$. By another result of Smith, $F$ is a rational sphere (see \cite{Smith38}). By induction, the induced $\gT^r$-action on $F$ has a codimension one torus $\gT^{r-1}$ with non-trivial fixed-point set $F^{\gT^{r-1}}$. Since $F^{\gT^{r-1}} = F \cap M^{\gT^{r-1}}$, this subtorus is the one we seek.
\end{proof}
}

	The second main tool from positive curvature we use is Wilking's Connectedness Lemma (see \cite{Wilking03}), which is the $k = 1$ statement of the following theorem. The generalization of the first part of (1) and of (2) to the case where $k \geq 2$ is stated in \cite[Remark 2.4]{Wilking03}. 
	For the second part of (1), the generalization to $k \geq 2$ was proved by the second author (see \cite{Mouille22b}), building on work of Guijarro and Wilhelm (see \cite{GuijarroWilhelm22}).
	
	\begin{theorem}[Connectedness Lemma]\label{thm:connprinc}
		Let $M^n$ be a closed, connected Riemannian manifold with $\pr{k}$. 
		\begin{enumerate}
			\item If $N^{n-d}$ is an embedded totally geodesic submanifold of $M^n$, then the inclusion 
			\[
				N^{n-d} \inj M^n \text{ is } (n-2d+2-k)\text{-connected}.
			\]
			Furthermore, if there is a Lie group $\gG$ acting on $M^n$ by isometries and fixing $N^{n-d}$ pointwise, then the inclusion is $(n-2d+2-k+\delta(\gG))$-connected, where $\delta(\gG)$ is the dimension of the principal orbits of the $\gG$-action on $M^n$.
			
			\item If $N_1^{n-d_1}$ and $N_2^{n-d_2}$ are embedded totally geodesic submanifolds with 
			   $d_1 \leq d_2$,
			then the intersection $N_1^{n-d_1} \cap N_2^{n-d_2}$ is also totally geodesic, and the inclusion 
			\[
				N_1^{n-d_1} \cap N_2^{n-d_2} \inj N_2^{n-d_2}  \text{ is } (n-d_1-d_2+1-k)\text{-connected}.
			\]
		\end{enumerate}
	\end{theorem}
	The Connectedness Lemma forces restrictions at the level of cohomology when combined with the following lemma, which is a topological result about highly connected inclusions of Poincar\'e duality spaces that was proved by Wilking \cite{Wilking03}:
	\begin{lemma}[Periodicity Lemma]\label{lem:periodic}
		Suppose $N^{n-d} \inj M^n$ is a $(n - d - l)$-connected inclusion of connected, closed, orientable manifolds. If $e \in H^d(M^n;\bZ)$ denotes the Poincar\'e dual of the image in $H_{n-d}(M^n;\bZ)$ of the fundamental class of $N$, then the homomorphisms 
		$\cup e: H^i(M;\bZ)\to H^{i+d}(M;\bZ)$ given by $x\mapsto x\cup e$ are surjective for $l\leq i<n-d-l$ and injective for $l<i\leq n-d-l$.
	\end{lemma}
	
	Of particular importance to us is the case in the Periodicity Lemma when $l = 1$ and $d = 2$.
	Based on whether $e$ is zero or non-zero, if $M^n$ is simply connected, we find that $M^n$ has the cohomology of $S^n$, $\CP^{\frac n 2}$, or more generally a finite connected sum $\CP^{\frac n 2} \# \ldots \# \CP^{\frac n 2}$ with at least two summands. For the first two of these manifolds, it is well known that homotopy rigidity is automatic. One has some level of rigidity in the non-simply connected case as well, according to the following (for a proof, see \cite[Theorem 3.4]{KennardSamaniSearle21preprint}):
	
	\begin{theorem}[Cohomology-to-homotopy Lemma]\label{thm:CohomologyToHomotopy}
	Let $M^n$ be a closed, smooth manifold. 
	The following hold:
		\begin{enumerate}
		\item If $\pi_1(M)$ is cyclic {(possibly trivial) } and the universal cover $\tilde M$ is a cohomology sphere, then $M$ is homotopy equivalent to $S^n$, $\RP^n$, or a lens space.
		\item If $\pi_1(M)$ is trivial and $M$ is a cohomology $\CP^{\frac n 2}$, then $M$ is homotopy equivalent to $\CP^{\frac n 2}$. 
		\end{enumerate}
	\end{theorem}
	
	Finally, to upgrade further from homotopy rigidity to homeomorphism or diffeomorphism rigidity, we use the following two results. 
	The first is for spheres and was proved by Montgomery and Yang \cite{MontgomeryYang67}, and the second is for complex projective spaces and was proved by Fang and Rong \cite{FangRong04}.
	
	\begin{theorem}
	[Diffeomorphism rigidity for spheres]
	\label{thm:homotopysphere}
		Suppose $M$ is a homotopy sphere, and assume the circle $\gS^1$ acts smoothly on $M$ such that the fixed-point set $N$ is simply connected and of codimension $2$. 
		Then $M$ is diffeomorphic to the standard sphere $S^n$ such that the $\gS^1$-action on $M$ is smoothly equivalent to a linear circle action on $S^n$.
	\end{theorem}
	
	\begin{theorem}
	[Homeomorphism rigidity for complex projective spaces]
	\label{thm:homotopyCPn}
		Suppose $M$ is a homotopy $\CP^n$, and assume a submanifold $N$ of codimension $2$ is homeomorphic to $\CP^{n-1}$.
		If the inclusion map $N \inj M$ is at least $3$-connected, then $M$ is homeomorphic to $\CP^n$.
	\end{theorem}

\section{Maximal symmetry rank for $\mathbf{Ric_2 > 0}$ in dimension 4}\label{sec:dim4}
	
	In this section, we establish the four dimensional case of Theorem \ref{result:k=2}.
	First, we survey the topology of the spaces in question without curvature considerations.
	
	Throughout this section, let $M$ be a closed, $4$-dimensional, simply connected, $\gT^2$-manifold, and let $M^*$ denote the orbit space $M/\gT^2$. 
	The orbit structure of such spaces were studied by Orlik and Raymond in \cite{OrlikRaymond70}, which we will summarize here.
	For the manifolds $M$ under consideration, the isotropy groups are connected, meaning possible isotropy groups are either trivial or isomorphic to $\gS^1$ or $\gT^2$; see Lemma 5.2 in \cite{OrlikRaymond70}.
	The orbit space $M^*$ is homeomorphic to a closed $2$-dimensional disk, and the boundary $\partial M^*$ consists of a cycle (graph) with the number of vertices equal to the Euler characteristic of $M$.
	We will assume an orientation on $M^*$, and hence on $\partial M^*$, and we will accordingly fix an enumeration for the vertices of $\partial M^*$: $f_0^*, f_1^* , \dots , f_{t-1}^*$, where $t = \chi(M)$.
	Each vertex $f_i^*$ in $\partial M^*$ corresponds to an isolated fixed point $f_i\in M$ of the $\gT^2$-action.
	Let $\Sigma_i^*$ denote the edge connecting $f_i^*$ to $f_{i+1}^*$ (counting mod $t$). 
	Points along $\Sigma_i^*$ correspond to $1$-dimensional orbits in $M$, all of which have the same isotropy group, which is isomorphic to $\gS^1$.
	In other words, $\Sigma_i^*$ corresponds to a $2$-dimensional sphere $\Sigma_i$ in $M$ that is fixed by a circle subgroup of $\gT^2$.
	Fixing a parametrization $(z_1,z_2)$ of $\gT^2 = \bR^2/\bZ^2$, each $\gS^1$ isotropy is equal to $\{(z_1,z_2):m z_1 + n z_2 = 0\}$ for some relatively prime integers $m$ and $n$.
	For a given $\gS^1$ isotropy, the associated vector $(m,n)\in \bZ^2$ is unique up to sign. 
	Given an edge $\Sigma_i^*$ of $\partial M^*$, we will call the vector $(m_i,n_i)$ that corresponds to the $\gS^1$ isotropy of the edge the \textit{weight} of $\Sigma_i^*$; see Figure \ref{fig:orbitspace}.
	
	\begin{figure}[ht]
		\includegraphics[width=0.4\textwidth]{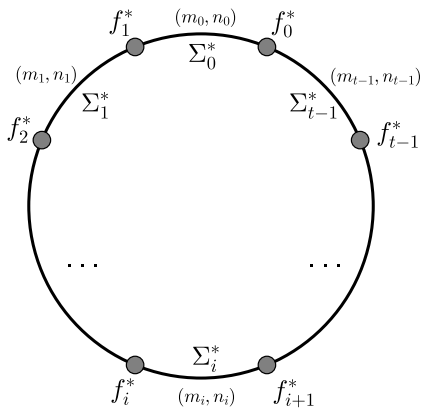}
		\caption{Weighted orbit space of closed simply connected $4$-dimensional $\gT^2$-manifold.}
		\label{fig:orbitspace}
		% Exporting from Geogebra: 1 unit = 5 cm
	\end{figure}
	
	Given two adjacent edges $\Sigma_{i-1}^*$ and $\Sigma_i^*$, the common vertex $f_i^*$ corresponds to a point of intersection $f_i$ between the two spheres $\Sigma_{i-1}$ and $\Sigma_i$ in $M$.
	Define the determinant of the weights of these edges
	\[
		\ep_i \defeq \det\begin{bmatrix} m_{i-1} & m_i \\ n_{i-1} & n_i \end{bmatrix}.
	\]
	Because the $\gS^1$ isotropies of these spheres must generate the homology of $\gT^2$, it follows that the determinant of the weights must satisfy $\ep_i = \pm 1$.
	
	More generally, given (not necessarily adjacent) edges $\Sigma_i^*$ and $\Sigma_j^*$ in $\partial M^*$, we will denote the determinant of their weights by
	\[
		r_{i,j} \defeq \det\begin{bmatrix} m_i & m_j \\ n_i & n_j \end{bmatrix}.
	\]
	Notice that $r_{i-1,i} = \ep_i = \pm 1$ for all $i$ (counting mod $t$) and that $r_{i,j} = 0$ for some $i$ and $j$ {if and only if } the corresponding spheres $\Sigma_i$ and $\Sigma_j$ in $M$ are fixed by the same circle subgroup of $\gT^2$.
	
	For the next few remarks, the isomorphisms mentioned are in the category of equivariant diffeomorphisms.
	If the number of fixed points $t = 2$, then $M \cong S^4$.
	If $t=3$ and $-\ep_0 \ep_1 \ep_2 = 1$ (resp. $-1$), then $M\cong \CP^2$ (resp. $\mCP{2}$).
	Note, when a parametrization of $\gT^2$ is specified, an orientation of $M^*$ determines an orientation of $M$, and vice versa.
	For the case $t=4$, $M\cong \CP^2 \# \CP^2$, $\mCP{2} \# \mCP{2}$, $S^2 \times S^2$, or $\CP^2 \# \mCP{2}$ depending on the values of $\ep_0, \ep_1, \ep_2, \ep_3, r_{0,2},$ and $r_{1,3}$.
	The conditions on $\ep_i$ and $r_{i,j}$ that determine $M$ for $t\leq 4$ are given in Table \ref{table:dim4} (see \cite[page 552]{OrlikRaymond70}):
	
	\begin{table}[H]
		\begin{tabular}{||c|c|l||}
			\hline
			$t$ & $M$ & \textit{Condition} \\
			\hline
			\hline
			$2$ & $S^4$ & \\
			\hline
			$3$ & $\CP^2$ & $-\ep_0 \ep_1 \ep_2 = +1$ \\
			& $\mCP{2}$ & $-\ep_0 \ep_1 \ep_2 = -1$ \\
			\hline
			$4$ &  $\CP^2 \# \CP^2$ & $-\ep_0 \ep_1 \ep_2 \ep_3 = +1$, and $r_{1,3} \in \{\ep_2 \ep_3, 2\ep_2 \ep_3\}$ \\
			& $\mCP{2} \# \mCP{2}$ & $-\ep_0 \ep_1 \ep_2 \ep_3 = +1$, and $r_{1,3} \in \{-\ep_2 \ep_3, -2\ep_2 \ep_3\}$ \\
			& $S^2 \times S^2$ & $-\ep_0 \ep_1 \ep_2 \ep_3 = -1$, and both $r_{0,2}$ and $r_{1,3}$ are even (at least one is $0$)\\
			& $\CP^2 \# \mCP{2}$ & $-\ep_0 \ep_1 \ep_2 \ep_3 = -1$, and either $r_{0,2}$ or $r_{1,3}$ is odd (the other one is $0$) \\
			\hline
		\end{tabular}
		\label{table:dim4}
		\caption{Equivariant diffeomorphism classes of simply connected $4$-dimensional $\gT^2$-manifolds with Euler characteristic $t \leq 4$.}
	\end{table}
	
	For the cases when $t\geq 5$, $M$ is equivariantly diffeomorphic to a connected sum of finitely many copies of $S^2\times S^2$, $\CP^2$, and $\mCP{2}$.
	In particular, for every pair of adjacent edges $\Sigma_{i-1}^*$ and $\Sigma_i^*$ in $\partial M^*$, there exists a third distinct edge $\Sigma_j^*$ such that
	\begin{enumerate}
		\item $r_{i,j} = \pm 1$ and $\Sigma_j^*$ is not adjacent to $\Sigma_i^*$, or \label{item:rij}
		\item $r_{i-1,j} = \pm 1$ and $\Sigma_j^*$ is not adjacent to $\Sigma_{i-1}^*$. \label{item:ri-1j}
	\end{enumerate}
	In the case of (\ref*{item:rij}), one can connect an interior point of edge $\Sigma_i^*$ to an interior point of edge $\Sigma_j^*$ using a simple curve $L^*$ through the interior of $M^*$.
	This curve $L^*$ separates $M^*$ into two disjoint regions, the closures of which we will denote by $X_1^*$ and $X_2^*$.
	For $k=1$ or $2$, consider $N_k^* \defeq X_k^*/\{L^* \sim \mathrm{pt.}\}$, i.e. the disk obtained from $X_k^*$ by identifying the portion of its boundary containing $L^*$ to a point.
	The edges of $\partial N_k^*$ inherit weights from corresponding edges on $\partial M^*$; see Figure \ref{fig:decomp}.
	Then $N_k^*$ corresponds to the orbit space $N_k/\gT^2$ of some closed, simply connected, $4$-dimensional $\gT^2$-manifold with $3 \leq \chi(N_k) \leq \chi(M) - 1$.
	The curve $L^*$ in $M^*$ corresponds to an invariant $3$-sphere $L$ in $M$.
	In particular, $M$ is equivariantly diffeomorphic to the connected sum $N_1 \# N_2$, where the gluing occurs along $L$.
	Case (\ref*{item:ri-1j}) above leads similarly to a decomposition of $M$ into a connected sum $N_1' \# N_2'$ for some $N_k'$ with $3 \leq \chi(N_k') \leq \chi(M) - 1$.
	
	\begin{figure}[ht]%[H]
		\includegraphics[width=\textwidth]{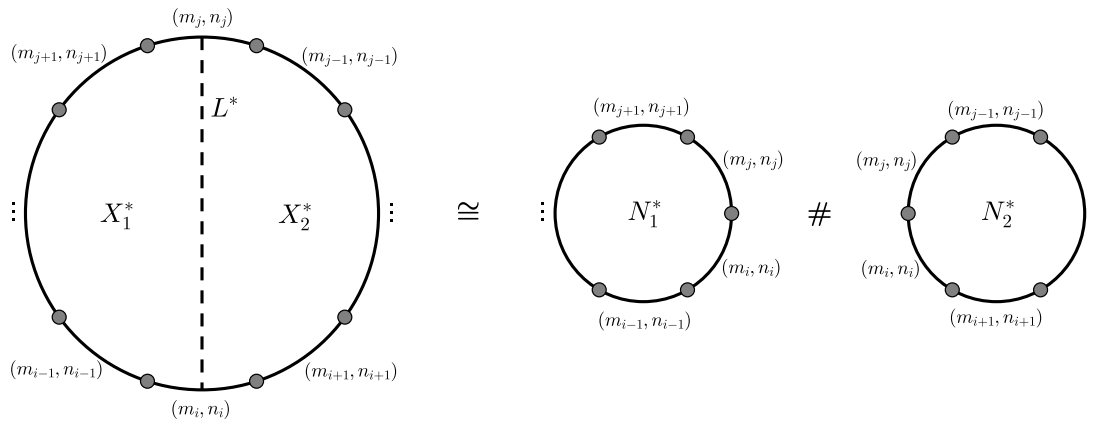}
		\caption{Decomposing the orbit space of a simply connected $4$-dimensional $\gT^2$-manifold with Euler characteristic $\geq 5$.}
		\label{fig:decomp}
		% during export from Geogebra, 1 unit = 5 cm
	\end{figure}

	%Notice that in $N_1^*$, the value of $\ep$ for the edges that meet at the point corresponding to $L^*$ (namely $\Sigma_i^*$ and $\Sigma_j^*$) will be opposite of the value of $\ep$ at $L^*$ in $N_2^*$.
	Notice that in $N_1^*$, the {$\epsilon$-value } for the edges that meet at the point corresponding to $L^*$ (namely $\Sigma_i^*$ and $\Sigma_j^*$) will be {negative } of the {$\ep$-value } at $L^*$ in $N_2^*$.
	Note that this type of decomposition of $M$ can also be carried out if $M \cong \CP^2 \# \CP^2$ or $\mCP{2} \# \mCP{2}$, but not for $S^2 \times S^2$ or $\CP^2 \# \mCP{2}$; c.f. Remark 5.10 in \cite{OrlikRaymond70}.
	Therefore, if $\chi(M) \geq 5$, then by repeating the procedure outlined above, $M^*$ can always be partitioned into finitely many pieces, each corresponding to $\CP^2$, $\mCP{2}$, or $S^2\times S^2$. 
	Furthermore, in such a decomposition $M \cong N_1 \# \dots N_m$, given a pair $N_i, N_{i+1}$ for $1\leq i\leq m-1$, the decomposition is done in such a way that the vertices on $\partial N_i^*$ and $\partial N_{i+1}^*$ at which the gluing $N_i \# N_{i+1}$ occurs have opposite signs for $\ep$.
	
	Now we establish curvature obstructions for the above spaces.
	In particular, our key observation is the following:
	
	\begin{lemma}\label{lem:disjointspheres}
		Let $M$ be a compact, connected, $4$-dimensional Riemannian manifold with $\pr{2}$.
		If $\gS^1$ acts effectively and by isometries on $M$ and fixes a $2$-dimensional submanifold $N$ pointwise, then $N$ must be connected.
	\end{lemma}
	
	\begin{proof}
		Since $\gS^1$ acts effectively on $M$, the principal orbits are $1$-dimensional. 
		By Theorem \ref{thm:connprinc}, the inclusion $N \inj M$ is $1$-connected.
		In particular, $N$ is connected since $M$ is.
	\end{proof}
	
	\begin{corollary} \label{cor:nospherebundles}
		Let $M$ be a compact, simply connected, $4$-dimensional Riemannian manifold with $\pr{2}$.
		If $\gT^2$ acts effectively and by isometries on $M$, and if $(m_i,n_i)$ and $(m_j,n_j)$ are weights for non-adjacent edges of $\partial M^*$, then {$r_{i,j} \neq 0$. }
%		\[
%			r_{i,j} = \det\begin{bmatrix} m_i & m_j \\ n_i & n_j \end{bmatrix} \neq 0.
%		\]
		In particular, neither $S^2 \times S^2$ nor $\CP^2 \# \mCP{2}$ admit a metric with $\pr{2}$ that is invariant under a $\gT^2$-action.
	\end{corollary}
	
	\begin{proof}
	{
	If some $r_{i,j} = 0$, then the weight vectors at the corresponding edges of $\partial M^*$ are parallel, the circle isotropy groups associated with these edges agree, and hence this circle has fixed-point set in $M$ consisting of at least two two-dimensional components, a contradiction to Lemma \ref{lem:disjointspheres}.
	}
	%The first statement follows from Lemma \ref{lem:disjointspheres} and the definition of weights.
		From Table \ref{table:dim4}, the orbit structure of $S^2 \times S^2$ and $\CP^2 \# \mCP{2}$ require that $r_{i,j} = 0$ for $(i,j) = (0,2)$ or $(1,3)$.
		Thus, these manifolds cannot support a metric with $\pr{2}$ that is invariant under a $\gT^2$-action.
	\end{proof}
	
	We can now establish the four-dimensional case of Theorem \ref{result:k=2}.
	
	\begin{theorem}\label{thm:dim4}
		Let $M$ be a closed, simply connected, $4$-dimensional Riemannian manifold with $\pr{2}$.
		If $\gT^2$ acts effectively and by isometries on $M$, then $M$ is equivariantly diffeomorphic to $\#_{i=1}^b \CP^2$ for some $b\geq 0$.
	\end{theorem}
	
	\begin{proof}
		Fix a parametrization of $\gT^2$ and an orientation of $M$, which then fixes an orientation of $M^*$.
		If $\chi(M) \leq 4$, then the only candidates are $S^4$, $\CP^2$, $S^2 \times S^2$, $\CP^2 \# \mCP{2}$, or $\CP^2 \# \CP^2$, up to a change in orientation.
		However, by Corollary \ref{cor:nospherebundles}, neither $S^2 \times S^2$ nor $\CP^2 \# \mCP{2}$ can admit an invariant metric with $\pr{2}$.
		
		If $\chi(M) \geq 5$, then following the procedure outlined in the beginning of this section, there exist non-adjacent edges of the boundary $\partial M^*$ of the orbit space whose respective weights have determinant $\pm 1$. 
		The orbit space $M^*$ can then be separated along a curve joining these two edges, and accordingly, $M$ decomposes as $N_1 \# N_2$ for some closed, simply connected, $4$-dimensional $\gT^2$-manifolds $N_k$ with $3 \leq \chi(N_k) \leq \chi(M) - 1$, for $k=1,2$.
		The weights of the edges of the boundaries $\partial N_k^*$ are inherited from the corresponding edges in $\partial M^*$, along with the orientations of their boundaries.
		This process can be repeated until $M$ is written as a connected sum $N_1 \# \dots \# N_m$ such that each $N_k$ if equivariantly diffeomorphic to $\CP^2$, $\mCP{2}$, or $S^2\times S^2$. 
		
		Because the weights of the edges of the boundaries $\partial N_k^*$ are inherited from edges of $\partial M^*$, if $\partial N_k^*$ has non-adjacent edges whose weights have determinant zero for some $k$, then so does $\partial M^*$.
		Thus, by Corollary \ref{cor:nospherebundles}, each space $N_k$ must be a complex projective space, and furthermore, they all have the same orientation.
	\end{proof}

\section{Maximal symmetry rank for $\mathbf{Ric_2 > 0}$ in dimension 6}\label{sec:dim6}
	
	In this section, we establish the six-dimensional case of Theorem \ref{result:k=2}. In contrast to the case of positive sectional curvature ($\pr 1$), the $\gT^3$-action on $M^6$ need not have a fixed point. Additionally, it does not follow immediately from the Connectedness and Periodicity Lemmas that the second Betti number satisfies $b_2(M) \leq 1$ as it does in the positive sectional curvature case. 
	The following is the case in which we can argue that $b_2(M) \leq 1$:
	
	\begin{proposition}\label{prop:dim6fixedpt}
		Let $M$ be a $6$-dimensional, closed, simply connected Riemannian manifold with $\pr{2}$.
		Suppose $\gT^3$ acts effectively and by isometries on $M$.
		If the $\gT^3$-action has a fixed point, then $M$ is diffeomorphic to $S^6$ or $\CP^3$.
	\end{proposition}
	
	\begin{proof}
		Suppose $p$ is a fixed point for the $\gT^3$-action on $M$.
		{
		Because the $\gT^3$-action on $M$ is effective, the isotropy representation of $\gT^3$ on the normal space to $p$ is faithful and hence has complex dimension at least three. In other words, $p$ is an isolated fixed point. At $p$, 
		}
		%Then at $p$, 
		the isotropy representation $\gT^3 \to \gU(3) \subset \gO(6)$ is of the form $(z_1,z_2,z_3) \mapsto \diag(z_1,z_2,z_3) \in \gU(3)$ for some choice of parametrization of $\gT^3$ and basis in $T_pM$.
		For each $i\in\{1,2,3\}$, define $\gS_i^1 \subset \gT^3$ to be the circle subgroup parametrized by $z_i$ in this representation, and let $N_i$ denote the four-dimensional fixed-point component $M_p^{\gS_i^1}$.

		By Theorem \ref{thm:connprinc}, the inclusions $N_i^4 \inj M^6$ are $3$-connected. This implies that $N_i$ is simply connected since $M$ is, that $H^2(N_i;\bZ) \cong H^2(M;\bZ)$ for all $i$, and that $H^3(M^6;\bZ) = 0$ since it injects into $H^3(N_i^4;\bZ)$, which is zero by Poincar\'e duality. 
		In particular, defining $b = b_2(M^6)$, we have $\chi(M^6) = 2 + 2b$ and $\chi(N_i^4) = 2 + b$ for all $i$.
		
		Furthermore, by Theorem \ref{thm:connprinc}, the inclusions $N_i \cap N_j \inj N_j$ are $1$-connected, and in particular, $N_i \cap N_j$ is connected for all $i\neq j$. 
		Also, each $2$-dimensional intersection $N_i \cap N_j$ is orientable by Lemma \ref{lem:torusfpc} and has an effective $\gS^1$ action with non-empty fixed-point set that contains $p$.
		Thus, each $N_i \cap N_j$ is a $2$-sphere, and $\chi(N_i \cap N_j) = 2$ for all $i\neq j$.
		
		Because $(N_1 \cup N_2 \cup N_3)^{\gT^3} \subseteq M^{\gT^3}$, by Lemma \ref{lem:Conner} we have
		\begin{align*}
		\chi(M) &\geq \chi(N_1 \cup N_2 \cup N_3) \\
		&= \sum_{i} \chi(N_i) - \sum_{i<j} \chi(N_i \cap N_j) + \chi(N_1 \cap N_2 \cap N_3)\\
		&= 3(2+b) - 3(2) + \chi(N_1 \cap N_2 \cap N_3).
		\end{align*}
		Since $\chi(M) = 2 + 2b$ and $N_1 \cap N_2 \cap N_3$ is a non-empty collection of isolated fixed points for the $\gT^3$-action on $M$, we have
		\[
			b_2(M) = b \leq 2 - \chi(N_1 \cap N_2 \cap N_3) \leq 1.
		\]
		It follows that $M$ has the homology groups of $S^6$ or $\CP^3$, and moreover by Lemma \ref{lem:periodic}, $M$ has the cohomology of one of these spaces. 
		Finally, Theorems \ref{thm:CohomologyToHomotopy}, \ref{thm:homotopysphere}, and \ref{thm:homotopyCPn} along with the classification of closed, simply connected $6$-manifolds imply that $M$ is diffeomorphic to $S^6$ or $\CP^3$.		
	\end{proof}
	
	To finish the proof of Theorem \ref{result:k=2} in dimension six, it suffices to consider the case where the torus action does not have a fixed point. 
	We seek to show that $\chi(M^6) = 0$ and, moreover, that $M^6$ is diffeomorphic to $S^3 \times S^3$ if the second Betti number of $M$ vanishes. 
	In the interest of potentially proving $M$ is diffeomorphic to $S^3 \times S^3$ without the assumption that $b_2(M)$ vanishes, we present the following partial progress:
	
	\begin{proposition}\label{prop:dim6nofixedpt}
		Let $M$ be a $6$-dimensional, closed, simply connected Riemannian manifold with $\pr{2}$.
		Suppose $\gT^3$ acts effectively and by isometries on $M$.
		If the $\gT^3$-action has no fixed points, then $\chi(M) = 0$. 
		
		Moreover, the $\gT^3$-action is not free, all non-trivial isotropy groups are isomorphic to $\gS^1$, and the singular $\gT^3$-orbits are isolated and diffeomorphic to $\gT^2$. 
		In particular, the orbit space $M^* = M/\gT^3$ is homeomorphic to $S^3$.
	\end{proposition}
	
	\begin{proof}
		The conclusion $\chi(M) = 0$ follows immediately from the equality $\chi(M) = \chi(M^{\gT^3})$ from Lemma \ref{lem:Conner}. 
		First, we claim that all isotropy groups must be connected.
		Otherwise, suppose $\Gamma \subset \gT^3$ is a finite isotropy group at a point $p \in M$.
		Then the torus $\gT^3/\Gamma \cong \gT^3$ acts effectively and by isometries on the totally geodesic fixed-point component $N = M_p^\Gamma$.
		Because the action is effective, $3 \leq \dim (N) \leq 5$.
		Since $N$ is totally geodesic in $M$, $N$ has $\pr{2}$.
		Then by the symmetry rank bound for $\pr{2}$, we have $3 \leq \lfloor \frac{\dim(N)+1}{2} \rfloor$, which implies that $\dim(N) = 5$.
		Then by Theorem \ref{thm:connprinc}, the inclusion $N \inj M$ is $4$-connected.
		Because $M$ is simply connected, so is $N$, and hence $N$ is orientable.
		Thus by Lemma \ref{lem:periodic}, $H^3(M;\bZ) \cong H^2(M;\bZ) \cong H^1(M;\bZ) \cong 0$.
		Hence, it follows from Poincar\'e duality and the Universal Coefficients theorem that $\chi(M) > 0$, which contradicts the hypothesis that the $\gT^3$-action on $M$ has no fixed points.
		Therefore, all isotropy groups must be connected.
		
		Next we claim that the components of the fixed-point set of any non-trivial isotropy group must be $2$-dimensional.
		Otherwise, there exists a connected isotropy group $\gT_p^3$ that fixes a connected submanifold $F$ of dimension $0$ or $4$.
		If $\dim(F) = 0$, then the induced action of $\gT^3$ on $F$ is trivial, and hence the $\gT^3$-action on $M$ has a fixed point, which is again a contradiction.
		If $\dim(F) = 4$, then because $F$ is fixed by $\gT_p^3$, which is isomorphic to $\gS^1$ or $\gT^2$, the inclusion $F \inj M$ is at least $3$-connected.
		Thus, $F$ is simply connected, and by Poincar\'e duality, has $\chi(F) > 0$.
		Hence, $\gT^3$ has a fixed point in $F \subset M$, which again is a contradiction.
		Therefore, the components of the fixed-point set of any non-trivial isotropy group must indeed be $2$-dimensional.
		
		It then follows from Part (\ref*{item:dimG>1}) of Lemma \ref{lem:torusfpc} that each non-trivial isotropy group is isomorphic to $\gS^1$.
		Thus, each singular orbit of the $\gT^3$-action is diffeomorphic to $\gT^2$ and coincides with a component of the fixed-point set of some $\gS^1$ isotropy group.
		Furthermore, given a point $p$ on a singular orbit, because circles are the only possible isotropy groups and components of their fixed-point sets are only $2$-dimensional, $\gT^2 \defeq \gT^3/\gS^1$ must act freely on the normal space to the singular orbit at $p$.
		In particular, the singular orbits are isolated.
		
		Since $M$ is simply connected and all $\gT^3$-orbits are connected, the orbit space $M^* = M/\gT^3$ is a simply connected $3$-manifold (see \cite[Corollary IV.4.7]{Bredon72}). 
		Because the $\gT^3$-action on $M$ only has $\gS^1$ isotropy groups whose fixed-point components are $2$-dimensional and isolated, it follows that $M^*$ has no boundary, and by the resolution to the Poincar\'e conjecture \cite{Perelman02preprint,Perelman03preprint1,Perelman03preprint2}, we have that $M^*$ is homeomorphic to $S^3$.
	\end{proof}

	We remark that Galaz-Garc\'ia and Searle show if $M^n$ is closed, simply connected, and has $\gT^{n-3}$-symmetry, if the orbit space $M^* = M^n/\gT^{n-3}$ is homeomorphic to $S^3$, and if all non-trivial isotropy groups are isomorphic to $\gS^1$, then $\pi_2(M^n) \cong \bZ^{s-n+2}$, where $s$ is the number of isolated singular orbits \cite[Proposition 4.5]{GalazGarciaSearle14}. 
	It then follows from \cite[Lemma 2-6]{GroveWilking14} that if $n = 6$ and $M^6$ has non-negative sectional curvature, then $s \leq 4$, and hence $\pi_2(M^6) \cong 0$ (see \cite[Proposition 4.12]{EscherSearle19}). 
	Escher and Searle then use these observations to prove such a manifold $M^6$ must be diffeomorphic to $S^3 \times S^3$ \cite[Proposition 4.13]{EscherSearle19}.
	Their conclusion relies on the fact that $M^*$ has non-negative curvature in the sense of Alexandrov geometry. 
	Since our condition of $\Ric_2 > 0$ allows for some negative sectional curvatures, we do not know whether it is possible to establish an upper bound on $s$ in our case.
	This leaves us with the following:
	
	\begin{remark}\label{rem:2-connected}
		If $M$ is a manifold as in Proposition \ref{prop:dim6nofixedpt}, then by \cite[Proposition 4.5]{GalazGarciaSearle14}, $\pi_2(M) \cong \bZ^{s-4}$, where $s$ is the number of isolated singular orbits. 
		If one could show that $s = 4$, then it would follow as in the proof of \cite[Proposition 4.13]{EscherSearle19} that $M$ is diffeomorphic to $S^3 \times S^3$.
		Such a result, along with those established here, would imply that the only closed, simply connected 6-manifolds with $\pr{2}$ and $\gT^3$-symmetry are $S^6$, $\CP^3$, and $S^3 \times S^3$.
	\end{remark}
	
	Finally, we include the following observation, as it is used in the proof of Theorem \ref{result:pi1} on the non-simply connected case.	
	
	\begin{corollary}\label{cor:fixedptsS3xS3}
		If $S^3 \times S^3$ is equipped with a metric having $\pr{2}$ that is invariant under an effective $\gT^3$-action, 
		then every non-trivial isotropy group is isomorphic to $\gS^1$, and the fixed-point set of any such $\gS^1$ is connected and diffeomorphic to $\gT^2$.
	\end{corollary}
	
	\begin{proof}
		In Proposition \ref*{prop:dim6nofixedpt}, we established that every non-trivial isotropy group is isomorphic to $\gS^1$, and the components of the fixed-point sets of of these $\gS^1$ isotropies must be isolated and diffeomorphic to $\gT^2$.
		Now given an arbitrary $\gS^1$ isotropy group, by Lemma \ref{lem:Conner}, we have 
		\[
			\sum b_i\of{(S^3\times S^3)^{\gS^1}} \leq \sum b_i(S^3 \times S^3) = 4.
		\]
		Therefore, $(S^3\times S^3)^{\gS^1}$ must consist of a single torus $\gT^2$.
	\end{proof}
	
\section{Maximal symmetry rank for $\mathbf{Ric_2 > 0}$ in dimensions $\mathbf{2n \geq 8}$}\label{sec:dim8+}
	
	In this section, we finish the proof of Theorem \ref{result:k=2} by induction. The result in dimension six is used to prove dimension eight, and this result is then used as our base for higher dimensions.
	
	\begin{theorem}\label{thm:dim8+}
		Let $M$ be closed, simply connected Riemannian manifold of even dimension $2n\geq 8$ with $\pr{2}$.
		If $\gT^n$ acts effectively and by isometries on $M$, then $M$ is either diffeomorphic to $S^{2n}$ or homeomorphic to $\CP^n$.
	\end{theorem}

\begin{proof}
	{
	We induct over the dimension $2n \geq 8$, and we prove the base case and the induction step simultaneously. 
	
	First, we claim that there exists a simply connected fixed-point component $N$ of a circle subgroup in $\gT^n$ such that $N$ has codimension two and such that the induced torus action on $N$ has a fixed point. All of this follows from Lemma 7.8 in \cite{Mouille22b}, but we include a direct argument here for completeness. 
	}
		Because $n\geq 4$, by Lemma \ref{lem:isotropyrank}, there exist circle subgroups of $\gT^n$ whose fixed-point sets are non-empty.
		Among all the circle subgroups and all components of their fixed-point sets, choose a subgroup $\gS^1$ and a component $N$ of its fixed-point set such that $N$ has maximal dimension.
		By Lemma \ref{lem:torusfpc}, $N$ is invariant under the action of $\gT^{n-1} = \gT^n / \gS^1$, and because $N$ was chosen to be maximal, the $\gT^{n-1}$-action on $N$ must be almost effective.
		Thus $\dim N \geq n-1 \geq 3$.
		{
		On the other hand, because 
		}
		$N$ is totally geodesic, it has $\pr{2}$, and hence the symmetry rank of $N$ is at most $\floor{\frac{\dim N + 1}{2}}$.
		Because $\dim M = 2n$ and $N$ has even codimension in $M$ (Lemma \ref{lem:torusfpc}), it follows that $\dim N = 2n-2$.
		Thus by Theorem \ref{thm:connprinc}, the inclusion $N \inj M$ is $(2n-3)$-connected, and thus $N$ is simply connected. 
		Note moreover that the odd Betti numbers of $M$ vanish by the Periodicity Lemma, and hence the odd Betti numbers of $N$ vanish by Lemma \ref{lem:Conner}.
		In particular, $\chi(N) > 0$ and hence the $\gT^{n-1}$-action on $N$ has a fixed point.
		
		%We now finish the proof.
		%We will now prove Theorem \ref{thm:dim8+} by induction.
		In summary, $N$ is a closed, $(2n-2)$-dimensional, simply connected manifold with $\pr{2}$ and maximal symmetry rank, and the induced $\gT^{n-1}$-action on $N$ has a fixed point.
		For the base case $2n=8$, it follows from Proposition \ref{prop:dim6fixedpt} that $N$ is diffeomorphic to $S^6$ or $\CP^3$.
		In the cases $2n\geq 10$, it follows from the induction hypothesis that $N$ is diffeomorphic to $S^{2n-2}$ or homeomorphic to $\CP^{n-1}$.
		Because $N \inj M$ is $(2n-3)$-connected, $M$ has the cohomology of $S^{2n}$ or $\CP^n$ up to degree $2n-3$, and it follows from Poincar\'e duality that $M$ is a cohomology $S^{2n}$ or $\CP^n$ (see \cite[Lemma 4.8.(1)]{KennardWiemelerWilking22preprint}). 
		Because $M$ is simply connected, $M$ is either a homotopy $S^{2n}$ or $\CP^n$ by Theorem \ref{thm:CohomologyToHomotopy}, and it follows from Theorems \ref{thm:homotopysphere} and \ref{thm:homotopyCPn} that $M$ is either {diffeomorphic } to $S^{2n}$ or {homeomorphic } to $\CP^n$.
%		
%		For the base case, assume $\dim M = 8$.
%		Then $N$ is a closed, $6$-dimensional, simply connected manifold with $\pr{2}$ and maximal symmetry rank. Moreover, the induced torus action on $N$ has a fixed point, so the proof in dimension six implies that $N^6$ is homeomorphic to $S^6$ or $\CP^3$. 
%		The Connectedness Lemma now implies that $M$ has the cohomology of $S^8$ or $\CP^4$ up to degree $5$, and it follows from Poincar\'e duality $M$ is a cohomology $S^8$ or $\CP^4$ (see \cite[Lemma 4.8.(1)]{KennardWiemelerWilking22preprint}). 
%		Because $M$ is simply connected, it then follows that $M$ is either a homotopy $S^8$ or $\CP^4$, and it follows from Theorems \ref{thm:homotopysphere} and \ref{thm:homotopyCPn} that $M$ is either {diffeomorphic } to $S^8$ or {homeomorphic } to $\CP^4$.
%		
%		For dimensions $2n\geq 10$, the result follows {in the same way except that we 
%		apply the induction hypothesis } since $N$ has $\dim N \geq 8$. 
		%must be either diffeomorphic to $S^{2n-2}$ or homeomorphic to $\CP^{n-1}$ by the induction hypothesis
		%{
		%. The claim for $M$ follows since 
		%} the inclusion $N \inj M$ is $(2n-3)$-connected.
	\end{proof}

\section{Maximal symmetry rank for $\mathbf{Ric_k > 0}$ with $\mathbf{k \geq 3}$}\label{sec:k>2}
	
	In this section, we prove Theorems \ref{result:k>2sphere} and \ref{result:k>2CP}. 
	First, we establish general results for manifolds with $\pr{k}$ for $k\geq 3$ that have maximal symmetry rank.
	The second author shows in \cite{Mouille22b} that any closed, connected, $n$-dimensional Riemannian manifold with $\pr{k}$ for some $k\in\{3,\dots,n-1\}$ has symmetry rank bounded above by $\floor{\frac{n+k}{2}} - 1$.
	Notice if $k=3$, or if $k=4$ and $n$ is odd, then this upper bound is equal to $\floor{\frac{n+1}{2}}$, which is the same bound as for positive sectional curvature or $\pr{2}$.
	Our first lemma applies the Isotropy Rank Lemma to show that manifolds with $\pr{k}$ and maximal symmetry rank often have circle actions with codimension-$2$ fixed-point components. {We will also need an analogous topological statement for rational spheres. These results are contained in the following: }
	
	\begin{lemma}\label{lem:codim2Rick}
		Let $M$ be a closed, $n$-dimensional Riemannian manifold equipped with an effective, isometric action by $\gT^r$. Assume one of the following:
			\begin{enumerate}
			\item[(a)] $M$ has $\pr k$ with $3 \leq k \leq n - 5$ and $r = \floor{\frac{n+k}{2}}-1$.
			\item[(b)] { $M$ is a rational homology sphere with $n \geq 2$ and $r \geq \floor{\frac{n+1}{2}}$.}
			\label{item:b}
			\end{enumerate}
		Then there exists a circle $\gS^1 \subset \gT^r$ 
		whose fixed-point set contains a component of codimension $2$ in $M$. 
		{
		In addition, the lower bound on $r$ in (b) is an equality.
		}
%		Let $M$ be a closed, $n$-dimensional Riemannian manifold with $\pr{k}$, and assume $3 \leq k \leq n-5$. 
%		If a torus $\gT^r$ of rank $r = \lfloor \frac{n+k}{2} \rfloor - 1$ acts effectively and by isometries on $M$, then there exists a circle $\gS^1 \subset \gT^r$ whose fixed-point set contains a component of codimension $2$ in $M$.
	\end{lemma}
	
	\begin{proof}
		{
		First we prove the lemma under the assumption of (a). 
		} 
		Because $n \geq k+5$, it follows from the equation $r = \lfloor \frac{n+k}{2} \rfloor - 1$ that $r \geq k+1$.
		Thus by Lemma \ref{lem:isotropyrank}, there exists a circle subgroup of $\gT^r$ with non-empty fixed-point set in $M$.
		Now among all the components of fixed-point sets for all circle subgroups of $\gT^r$, choose a component $F^f$ whose dimension $f$ is maximal, and let $\gS^1$ denote a circle that fixes $F$.
		Because the $\gT^r$ action on $M$ is effective, and since the codimension of $F$ is even (Lemma \ref{lem:torusfpc}), we must have $f \leq n-2$.
		By Lemma \ref{lem:torusfpc}, since the dimension of $F$ is maximal, the action of $\gT^{r-1} \defeq \gT^r / \gS^1$ on $F$ must be almost effective.
		Thus, $f \geq r-1 = \lfloor \frac{n+k}{2} \rfloor - 2$.
		In particular, since $n \geq k+5$, we have $f \geq \lfloor \frac{n+k}{2} \rfloor - 2 \geq \lfloor \frac{2k+5}{2} \rfloor - 2 =  k$.
		Now if $f = k$, then the constraints $\lfloor \frac{n+k}{2} \rfloor = \lfloor \frac{2k+5}{2} \rfloor$ and $n \geq k+5$ imply that $n = k+5 = f+5$, which contradicts the fact that $F$ has even codimension in $M$.
		Hence, we must have $f \geq k+1$, and because $F$ is totally geodesic in $M$, $F$ has $\pr{k}$.
		Then by the
		{
		Maximal Symmetry Rank bound applied to $F$, we have $r-1 \leq \lfloor \frac{f+k}{2} \rfloor -1$.
		Combining this with the assumption $r = \floor{\frac{n+k}{2}}-1$ implies $n+k \leq f+k+3$.
		}
		Therefore, because $n \equiv f \bmod 2$ and $f \leq n-2$, it follows that $f = n-2$.
		
	{
	Now we prove the lemma under the assumption of (b).
	By Lemma \ref{lem:isotropyrank_spheres}, 
	there exists a subgroup $\gT^{r-\delta}$ with non-empty fixed-point set,
	where $\delta$ is zero if $n$ is even and one if $n$ is odd.
	Let $F$ be a fixed-point component of $\gT^{r-\delta}$.
	Since the action is effective, the isotropy representation at a normal space $\nu_p F$ to $F$ is faithful.
	Hence $r - \delta \leq \tfrac 1 2 \codim F$.
	In addition, $\tfrac 1 2 \codim F \leq \tfrac{n-\delta}{2}$ by Lemma \ref{lem:torusfpc}.
	By the assumption $r \geq \floor{\frac{n+1}{2}}$, it follows that $r - \delta = \tfrac 1 2 \codim F= \tfrac{n-\delta}{2}$.
	In particular, $r = \floor{\frac{n+1}{2}}$.
	In addition, there are exactly $r - \delta$ irreducible subrepresentations, and 
	the representation is equivalent to the map sending $(z_1,\ldots,z_{r-\delta}) \in \gT^{r-\delta}$ 
	to $\diag(z_1,\ldots,z_{r-\delta}) \in \gU(r-\delta) \subseteq \SO(\nu_p F)$.
	As in the proof of Proposition \ref{prop:dim6fixedpt}, 
	we can intersect the kernels of any $r-\delta-1$ of the irreducible subrepresentations
	to obtain the desired circle.
	}
	\end{proof}
	
	Next, we establish a rigidity result for highly connected manifolds with $\pr{k}$ that have a codimension-two circle fixed-point component.
	
	\begin{proposition}\label{prop:k-1connected}
		Let $M$ be a closed, $n$-dimensional Riemannian manifold with $\pr{k}$. 
		Assume that {$3 \leq k \leq n - 3$ 
		and moreover that $k$ is odd if $k = \tfrac{n}{2}$. %$n = 6$ if $k = \tfrac{n}{2}$. 
		}
		If $M$ is $(k-1)$-connected,
		and if $\gS^1$ acts effectively and by isometries on $M$ such that 
		its fixed-point set contains a component $N$ of codimension $2$ in $M$, 
		then $M$ is diffeomorphic to $S^n$. 
	\end{proposition}
	
	\begin{proof}
		{
		First assume $3 \leq k \leq \tfrac{n-1}{2}$. 
		}
		By the Connectedness Lemma, the inclusion $N \inj M$ is $(n-k-1)$-connected. 
		In particular, $N$ is simply connected since $n - k - 1 \geq 2$.
		We claim that $M^n$ is a cohomology sphere. 
		Given the claim, Theorems \ref{thm:CohomologyToHomotopy} and \ref{thm:homotopysphere} imply that $M$ is diffeomorphic to $S^n$.
		To prove the claim, we apply the Periodicity Lemma to the inclusion $N \inj M$, which has codimension two and is $(n - k - 1)$-connected. 
		We then have $e \in H^2(M)$ that induces periodicity from degree $k - 1$ to degree $n - (k - 1)$. That is, the map $H^i(M) \to H^{i+2}(M)$ induced by multiplication by $e$ is surjective for $k-1 \leq i < n - (k-1) - 2$ and injective for $k - 1 < i \leq n - (k-1) - 2$. 
		Because $M$ is $(k-1)$-connected and $k \geq 3$, we have $e = 0$, so combining with the injectivity property implies that $H^i(M) = 0$ for all $0 < i \leq n - (k-1) - 2$. 
		Since $n - (k-1) - 2 \geq \tfrac{n-1}{2}$, Poincar\'e duality implies that $M$ is cohomology sphere, as claimed.
		
		{
		Second assume $\tfrac{n+1}{2} \leq k \leq n - 3$. 
		The condition that $M$ is $(k-1)$-connected implies that 
		$M$ is a homology sphere by Poincar\'e duality.
		Since $M$ is simply connected, we see as in the previous case that 
		$N$ is simply connected and hence that $M$ is diffeomorphic to $S^n$ by 
		Theorems \ref{thm:CohomologyToHomotopy} and \ref{thm:homotopysphere}.
		}
		
		{
		%Finally assume $k = \tfrac n 2$ and $n = 6$.
		Finally assume $k = \tfrac n 2$ and that $k$ is odd.
		%By the Connectedness Lemma, $N\inj M$ is $2$-connected, 
		By the Connectedness Lemma, $N\inj M$ is $(k-1)$-connected, 
		%and hence $N$ is simply connected.
		and hence $N$ is $(k-2)$-connected.
		%Thus by Poincar\'e duality, $b_3(N) = b_1(N) = 0$, and hence $\chi(N) = 2 + b_2(N)$. 
		Thus by Poincar\'e duality, $\chi(N) = 2 + b_{k-1}(N)$.
		%Because $M$ is $2$-connected, 
		Because $M$ is $(k-1)$-connected, 
		it follows from the estimate on the sum of even Betti numbers (Lemma \ref{lem:Conner})
		%that $N$ is the only connected component of $M^{\gS^1}$ and $b_2(N) = 0$.
		that $N$ is the only connected component of $M^{\gS^1}$ and $b_{k-1}(N) = 0$.
		%Thus, $2 = \chi(N) = \chi(M^{\gS^1}) = \chi(M) = 2 - b_3(M)$, 
		Thus, $2 = \chi(N) = \chi(M^{\gS^1}) = \chi(M) = 2 - b_k(M)$, 
		%and hence $b_3(M) = 0$.
		and hence $b_k(M) = 0$.
		Therefore, $M$ and $N$ are both simply connected cohomology spheres, 
		and by Theorems \ref{thm:CohomologyToHomotopy} and \ref{thm:homotopysphere},
		%$M$ is diffeomorphic to $S^6$. 
		$M$ is diffeomorphic to $S^n$. 
		}
		\end{proof}
	
	Now we prove our second main result:
	
	\begin{theorem}[Theorem \ref{result:k>2sphere}]
	Let $M^n$ be a $(k-1)$-connected, closed Riemannian manifold with $\Ric_k > 0$ for some $k \geq 3$. If $n \neq 7$ and if $M^n$ admits an effective, isometric $\gT^r$-action with $r = \floor{\frac{n+k}{2}} - 1$, then one of the following occurs:
		\begin{enumerate}
		\item $M$ is diffeomorphic to $S^n$ and $k \leq 4$, with equality only if $n$ is odd.
		\item $M$ is diffeomorphic to $S^3 \times S^3$ and $k = 3$.
		\end{enumerate}
	\end{theorem}
	
	\begin{proof}
		{
		The claim on $k$ in (1) follows as soon as we know that $M$ is a rational sphere by combining the assumption $r = \floor{\frac{n+k}{2}} - 1$ with the upper bound $r \leq \floor{\frac{n+1}{2}}$ from Lemma \ref{lem:codim2Rick}.
		}
		The claim on $k$ in (2) follows from the assumption that $M$ is $(k-1)$-connected.
		It suffices to prove the diffeomorphism claims.
		
	First, assume $k < \tfrac{n}{2}$. 
	Since $n \geq 2k + 1 \geq 7$ and $n \neq 7$, we have $k \leq n - 5$. 
	Hence Lemma \ref{lem:codim2Rick} implies the existence of a circle $\gS^1$ 
	with fixed-point set of codimension two, and Proposition \ref{prop:k-1connected} implies 
	$M$ is diffeomorphic to $S^n$. 
		
	Second, assume $k > \frac n 2$. 
	By the assumption that $M$ is $(k-1)$-connected and Poincar\'e duality, 
	we see that $M$ is a homology sphere.
	%Notice in particular that $k \leq n - 3$ since otherwise the upper bound $r \leq \floor{\tfrac{n+1}{2}}$ from Lemma \ref{lem:codim2Rick} is contradicted.
	{
	Since $k \geq 3$, the assumption on $r$ implies $r \geq \floor{\frac{n+1}{2}}$.		
	By Lemma \ref{lem:codim2Rick}, there exists a subgroup $\gS^1$ having
	a fixed-point component $N$ with codimension two.
	If $k \leq n - 3$, then Proposition \ref{prop:k-1connected} implies again that $M$ is diffeomorphic to $S^n$. If $k \geq n - 2$, then we apply Lemma \ref{lem:codim2Rick} once more to conclude that $r = \floor{\tfrac{n+1}{2}}$ and hence that $n \leq 6$. In this range, simply connected homology spheres are diffeomorphic to standard spheres, so again $M$ is diffeomorphic to $S^n$.
	}
	%By Lemma \ref{lem:k-1connected}, $M$ is diffeomorphic to $S^n$.
	%It follows for topological reasons that $\gT^r$ has a fixed point in even dimensions 
	%and that a codimension one subtorus $\gT^{r-1}$ has a fixed point in odd dimensions.
	%Indeed, this follows immediately in even dimensions since $\chi(M) = 2$, 
	%and it follows by a straightforward induction argument in odd dimensions 
	%using the classical fact proved by Smith that $\gT^2$ cannot act freely 
	%on a homology sphere; see \cite[Chapter III Theorem 8.1]{Bredon72}.
	%In particular, $r \leq \tfrac{n+1}{2}$, and 
	%there exists a fixed-point component $N$ of a circle in $\gT^r$ with codimension two. 
	%Smith's theorem implies that $N$ is a homology sphere, 
	%and the 
	%The Connectedness Lemma implies that $N$ is simply connected. 
	%Therefore $M$ is diffeomorphic to $S^n$ by 
	%Theorems \ref{thm:CohomologyToHomotopy} and \ref{thm:homotopysphere}.
		
		{
		Third, assume that $k = \tfrac n 2$ and that the fixed-point set of $\gT^r$ is non-empty.
		}
		For any fixed-point component $F^f$ and any $p \in F^f$, 
		the isotropy representation $\gT^r \to \SO(\nu_p F^f)$
		on the normal space to $F^f$ is faithful since the action on $M$ is effective. 
		{
		Hence $r \leq \tfrac{n-f}{2}$. 
		Since $f \geq 0$, the lower bound on $r$ implies $k = 3$ and $r = 3$. 
		}
		As in the proof of Proposition \ref{prop:dim6fixedpt}, 
		there exists a circle $\gS^1$ with a fixed-point component $N^4$ of codimension two. 
		{
		By Proposition \ref{prop:k-1connected},
		$M$ is diffeomorphic to $S^6$. 
		}
		
		{
		Finally, assume that $k = \tfrac n 2$ and that fixed-point set of $\gT^r$ is empty.
		Hence $\chi(M) = \chi(M^{\gT^r}) = 0$.
		}
		On the other hand, $\chi(M) = 2 + (-1)^k b_k(M)$ since $M^{2k}$ is $(k-1)$-connected.
		So it follows that $k$ is odd and $b_k(M) = 2$. 
		{
		If $k \geq 5$, then 
		}
		Lemma \ref{lem:codim2Rick} applies since $n = 2k \geq k + 5$, and we get 
		a circle $\gS^1$ whose fixed-point set has a component $N^{2k-2}$ of codimension two.
		{
		But then Proposition \ref{prop:k-1connected} implies that $M$ is a sphere,
		which contradicts the fact that $b_k(M) = 2$, so we must have $k = 3$.
		}
		Corollary 2.6 in \cite{EscherSearle19} now implies that $M$ is diffeomorphic to $S^3 \times S^3$.				
	\end{proof}
		
	Next we apply Lemma \ref{lem:codim2Rick} to prove our third main result.
	Recall that $M$ being an integral cohomology $\CP$ up to degree $k + 2$ means that $H^1(M;\bZ) \cong 0$, $H^2(M;\bZ)\cong \bZ$, and the map $H^i(M;\bZ) \to H^{i+2}(M;\bZ)$ induced by multiplication by a generator $x\in H^2(M;\bZ)$ is surjective for $0 \leq i < k$ and injective for $0 < i \leq k$.
	
	\begin{theorem}[Theorem \ref{result:k>2CP}]\label{thm:k>2CP}
		Fix $k \geq 3$, and let $M^n$ be a simply connected, closed Riemannian manifold.
		Assume further that $M^n$ is an integral cohomology $\CP$ up to degree $k + 2$. 
		If $M$ has $\Ric_k > 0$ and admits an effective, isometric $\gT^r$-action with $r = \floor{\frac{n+k}{2}} - 1$, then $n$ is even, $M^n$ is homeomorphic to $\CP^{\frac n 2}$, and $k = 3$.
	\end{theorem}
	
	\begin{proof}
		First, we claim it suffices to prove that $M$ has the cohomology of $\CP^{\frac n 2}$ in all degrees by Theorems \ref{thm:CohomologyToHomotopy} and \ref{thm:homotopyCPn}. 
		Indeed, if $M$ is a cohomology $\CP^{\frac n 2}$, it follows that $\chi(M) > 0$, that $\gT^r$ has a fixed point, and hence that $r \leq \tfrac{n}{2}$. 
		From the equation $r = \floor{\frac{n+k}{2}}-1$, since $n$ is even and $k\geq 3$, we have $k = 3$ and $r = \frac n 2$. 
		As in the proof of Proposition \ref{prop:dim6fixedpt}, the isotropy representation $\gT^r \to \gU(r)$ at a fixed point of the $\gT^r$-action is of the form $(z_1,\dots,z_r) \mapsto \diag(z_1,\dots,z_r)$ for a certain choice of coordinates.
		For each $j\in\{1,\dots,r\}$, let $\gS^1_j \subset \gT^r$ denote the circle subgroup parametrized by $z_j$, and let $N_j$ denote the fixed-point set of $\gS^1_j$, which is $(n-2)$-dimensional.
		Defining $F^{2i} = \bigcap_{j=1}^{r-i} N_j$ for each $i\in\{2,\dots,r-1\}$, we have a chain of inclusions 
		\[
			F^4 \subset F^6 \subset \ldots \subset F^{n-2} \subset M^n.
		\]
		Because each space $F^{2i}$ is {a fixed-point component of } a circle action on the subsequent space in the chain, it follows from \cite{Su63} that each $F^{2i}$ is a cohomology $\CP^{2i}$, and the generator of $H^2(F^{2i};\bZ)$ restricts to a generator of $H^2(F^{2i-2};\bZ)$ for all $i$. 
		In particular, each inclusion induces isomorphisms on cohomology in all degrees less than the dimension of the submanifold. 
		By the universal coefficients theorem and Hurewicz's theorem, it follows that each inclusion is $3$-connected. 
		Now $F^4$ is then homeomorphic to $\CP^2$ by Freedman's classification in dimension four, so we can apply Theorems \ref{thm:CohomologyToHomotopy} and \ref{thm:homotopyCPn} to conclude that $M^n$ is homeomorphic to $\CP^{\frac n 2}$. 
		
		We now proceed to the proof that $M$ has the cohomology of $\CP^{\frac n 2}$ in all degrees.
		Since we already know the integral cohomology is correct in degrees up to $k + 2$, the rest follows by Poincar\'e duality if $k + 2 \geq \tfrac{n+3}{2}$ (for proof of a similar fact in rational cohomology, see \cite[Lemma 4.8.(1)]{KennardWiemelerWilking22preprint}).
		We may therefore assume $k \leq \tfrac{n-2}{2}$.
		In particular, we have $n \geq 2k + 2 \geq 8$, and hence $k \leq \tfrac{n-2}{2} < n - 4$. 
		Lemma \ref{lem:codim2Rick} therefore implies the existence of a circle $\gS^1$ containing a fixed-point component $N^{n-2}$ of codimension two. 
		By Theorem \ref{thm:connprinc}, the inclusion $N \inj M$ is $(n-k-1)$-connected,
		and by Lemma \ref{lem:periodic}, there exists $e\in H^2(M)$ such that the homomorphism $\cup e:H^i(M) \to H^{i+2}(M)$ is surjective for $k-1 \leq i < n-k-1$ and injective for $k-1 < i \leq n-k-1$.
		Because $M$ is an integral cohomology $\CP$ up to degree $k+2 \geq 5$, if $x\in H^2(M) \cong \bZ$ denotes a generator, then $e = \lambda x$ for some $\lambda \in \bZ$.
		We will show that $\lambda = \pm 1$.
		
		Define $l \defeq \lfloor \frac{k}{2} \rfloor$.
		Then taking $i = 2l$ above, we have that $\cup (\lambda x):H^{2l}(M) \to H^{2l+2}(M)$ is an isomorphism if $k$ is even and an epimorphism if $k$ is odd.
		However, because $M$ is an integral cohomology $\CP$ up to degree $k+2$, it follows that this map is in fact an isomorphism $\bZ \to \bZ$ in either case.
		Thus, since $x^l$ and $x^{l+1}$ are a generators of $H^{2l}(M)$ and $H^{2l+2}(M)$, respectively, it follows that $\lambda x^{l+1} = \pm x^{l+1}$, and hence $\lambda = \pm 1$.
		Therefore, $M$ has the cohomology of $\CP^{n/2}$, and the result follows. 
	\end{proof}

\section{Fundamental groups for maximal symmetry rank}\label{sec:pi1}

	In this section, we prove our last main result, which deals with fundamental groups of manifolds with $\pr{2}$ and maximal symmetry rank:
	
	\begin{theorem}[Theorem \ref{result:pi1}]
		Let $M^n$ be a closed, connected Riemannian manifold with $\pr{2}$ and $\gT^r$ symmetry with $r = \left\lfloor \frac{n+1}{2} \right\rfloor$. If $\pi_1(M)$ is non-trivial, then one of the following occurs:
		\begin{enumerate}
			\item $M$ is homotopy equivalent to $\RP^n$ or a lens space, or
			\item $M$ has dimension six and, if additionally the universal cover is $S^3 \times S^3$, then $\pi_1(M) \cong \bZ_l \times \bZ_m$ for some $l,m\geq 1$.
		\end{enumerate}
	\end{theorem}
	
	\begin{proof}
		We may assume $n \geq 3$, since otherwise the condition $\pr{2}$ is vacuous. 
		We set $\Gamma = \pi_1(M)$ and pullback the metric and the torus action to the universal cover $\tilde M$. We get a $\gT^r$-action on $\tilde M$ that commutes with the free action of $\Gamma$ on $\tilde M$ by deck transformations. 
		Note that $\Gamma$ is finite by Myers' theorem.
		
		Following the proof{s of Theorem 1.1 in \cite{Mouille22b} and Theorem \ref{result:k=2} } in the simply connected case, we arrive at one of the following situations:
		\begin{enumerate}[(i)]
			\item There exists $\gS^1 \subseteq \gT^r$ such that the fixed-point set $\tilde M^{\gS^1}$ has a unique component $N^{n-2}$ with codimension two. Moreover, $\tilde M$ is a cohomology $S^n$ or $\CP^{\frac n 2}$. \label{item:SnCPn}
			
			\item The universal cover is $S^3 \times S^3$, and the torus $\gT^3$ contains a circle $\gS^1$ whose fixed-point set is connected and diffeomorphic to $\gT^2$. \label{item:S3xS3}
			
			\item The universal cover is $\CP^2 \# \ldots \# \CP^2$, and the torus $\gT^2$ contains $t = \chi(M)$ distinct isotropy groups $\gS^1_1,\ldots, \gS^1_t$ whose fixed-point sets have a unique $S^2$ component.\label{item:CP2connsum}
		\end{enumerate}
		
		Indeed, {Lemma 6.2 in \cite{Mouille22b} implies that (i) holds in the odd-dimensional case. For the even-dimensional case, } if the $\gT^r$ action has no fixed point, then (\ref*{item:S3xS3}) holds by Corollary  \ref{cor:fixedptsS3xS3}. 
		If instead the $\gT^r$ action has a fixed point, then the existence statements of (\ref*{item:SnCPn}) and (\ref*{item:CP2connsum}) were established in the proofs of the simply connected cases and the uniqueness statements follow from the generalizations of Frankel's theorem provided by Part 2 of Theorem \ref{thm:connprinc} and Lemma \ref{lem:disjointspheres}, respectively.
		
		Suppose first we are in Case (\ref*{item:SnCPn}) and that $\tilde M$ is a cohomology $S^n$, then $N$ is a cohomology $S^{n-2}$ by the sum of Betti numbers estimate (Lemma \ref{lem:Conner}). 
		Since the actions by $\Gamma$ and $\gT^r$ commute, $\Gamma$ acts freely on both $S^n$ and $S^{n-2}$. It follows that $\Gamma$ is cyclic (see \cite[Lemma 1.8]{FrankRongWang13}) in general and moreover $\bZ_2$ if $n$ is even. By \cite[Theorem 3.4]{KennardSamaniSearle21preprint}, it follows that $M$ is homotopy equivalent to real projective space or a lens space, as required.
		
		Next suppose we are in Case (\ref*{item:SnCPn}) and that $\tilde M$ is a cohomology $\CP^{\frac n 2}$. By \cite[Theorem 7.2]{Su63}, $N$ is a cohomology $\CP^{\frac{n}{2} - 1}$. Once again, $\Gamma$ acts freely on both of these manifolds. In particular, the order of $\Gamma$ divides both of their Euler characteristics. Since these differ by one, $\Gamma$ is trivial, a contradiction. %Hence this case does not occur.
		
		Next suppose we are in Case (\ref*{item:S3xS3}). 
		Let $\gS^1$ be a non-trivial isotropy group in $\gT^3$ whose fixed-point set $F \defeq M^{\gS^1}$ consists of a unique $2$-dimensional torus. 
		Fix any $x \in F$, and consider the diffeomorphism $\gT^2 \defeq \gT^3/\gS^1 \to F$ given by $g \mapsto g \cdot x$. 
		Using the inverse of this map, we obtain another a map 
		$\Gamma \to \gT^2$ denoted by $\gamma \mapsto g_\gamma$ and determined by the property that $\gamma\cdot x = g_\gamma \cdot x$.
		We claim that $\Gamma \to \gT^2$ is a group homomorphism.
		Given $\alpha, \beta \in \Gamma$, we find that
		\[
			g_{\alpha\beta} \cdot x = (\alpha\beta)\cdot x = \alpha\cdot(\beta\cdot x) = \alpha \cdot (g_\beta \cdot x) = g_\beta \cdot (\alpha \cdot x) = g_\beta\cdot(g_\alpha\cdot x) = (g_\alpha g_\beta)\cdot x,
		\]
		where we have used that $\gT^2$ is abelian and that the $\gT^2$- and $\Gamma$-actions commute. 
		By the injectivity of the map $\gT^2 \to F$, we find that $g_{\alpha \beta} = g_\alpha g_\beta$, and hence the map $\Gamma \to \gT^2$ is indeed a group homomorphism. 
		Since $\Gamma$ acts freely, this map is an injection. 
		Hence $\Gamma$ may be regarded as a subgroup of $\gT^2$, and it follows then that $\Gamma$ is either cyclic or a two-fold product of cyclic groups.
		
		Finally, suppose we are in Case (\ref*{item:CP2connsum}). 
		It suffices to prove that $\chi(M) = 2$, since then $M = S^4$ and we are in the situation of (\ref*{item:SnCPn}). 
		We suppose then that $\chi(M) \geq 3$ and seek a contradiction. 
		After possibly relabeling, we may assume that the first two circles, $\gS^1_1$ and $\gS^1_2$, have the property that their respective $S^2$ fixed-point components contain $\{f_0,f_1\}$ and $\{f_1,f_2\}$, respectively, where $f_0$, $f_1$, and $f_2$ are distinct isolated fixed points of the $\gT^2$-action (see Section \ref{sec:dim4}). 
		Since the free $\Gamma$-action commutes with the $\gT^2$-action, $\Gamma$ acts freely on both of the sets $\{f_0,f_1\}$ and $\{f_1,f_2\}$, which contradicts the assumption that $\Gamma$ is non-trivial.
	\end{proof}

\bibliographystyle{alpha}
\bibliography{bibfile}

\begin{thebibliography}{GWY19}

\bibitem[AQZ]{AmannQuastZarei20preprint}
Manuel Amann, Peter Quast, and Masoumeh Zarei.
\newblock The flavour of intermediate Ricci and homotopy when studying submanifolds of symmetric spaces.
\newblock preprint, arXiv:2010.15742.

\bibitem[Ber66]{Berger66}
Marcel Berger.
\newblock Trois remarques sur les vari\'et\'es riemanniennes \`a courbure
  positive.
\newblock {\em C. R. Acad. Sci. Paris S\'er. A-B}, 263:A76--A78, 1966.

\bibitem[Bre72]{Bredon72}
Glen~E. Bredon.
\newblock {\em Introduction to Compact Transformation Groups}, volume~46.
\newblock Academic Press, 1972.

\bibitem[Bur19]{Burdick19}
Bradley~Lewis Burdick.
\newblock Ricci-positive metrics on connected sums of projective spaces.
\newblock {\em Differential Geom. Appl.}, 62:212--233, 2019.

\bibitem[Bur20]{Burdick20}
Bradley~Lewis Burdick.
\newblock Metrics of positive {Ricci} curvature on the connected sums of
  products with arbitrarily many spheres.
\newblock {\em Ann. Global Anal. Geom.}, 58(4):433--476, 2020.

\bibitem[Cha19]{Chahine19}
Yousef~K. Chahine.
\newblock Volume estimates for tubes around submanifolds using integral
  curvature bounds.
\newblock {\em J. Geom. Anal.}, 2019.

\bibitem[Con57]{Conner57}
P.~E. Conner.
\newblock On the action of the circle group.
\newblock {\em Michigan Math. J.}, 4(3):241--247, 1957.

\bibitem[CW]{CrowleyWraith20preprint}
Diarmuid Crowley and David Wraith.
\newblock Intermediate curvatures and highly connected manifolds.
\newblock preprint, arXiv:1704.07057.

\bibitem[Dav]{Davis98preprint}
James~F. Davis.
\newblock ${S}_3 \times {S}_3$ acts freely on ${S}^3 \times {S}^3$.
\newblock preprint, arXiv:math/9806027.

\bibitem[Dea11]{Dearricott11}
Owen Dearricott.
\newblock A 7-manifold with positive curvature.
\newblock {\em Duke Math. J.}, 158(2), 2011.

\bibitem[DGM22]{DVGAM22}
Miguel {Dom\'inguez-V\'azquez}, David {Gonz\'alez-\'Alvaro}, and Lawrence
  Mouill\'e.
\newblock Infinite families of manifolds of positive $k^{\rm th}$-intermediate
  {Ricci} curvature with $k$ small.
\newblock {\em Math. Ann.}, 2022.

\bibitem[DGR]{DVGARVinprep}
Miguel {Dom\'inguez-V\'azquez}, David {Gonz\'alez-\'Alvaro}, and Alberto
  {Rodr\'iguez-V\'azquez}.
\newblock Normal homogeneous spaces with positive second intermediate {Ricci}
  curvature.
\newblock in preparation.

\bibitem[ES19]{EscherSearle19}
Christine Escher and Catherine Searle.
\newblock Non-negatively curved 6-manifolds with almost maximal symmetry rank.
\newblock {\em J. Geom. Anal.}, 29(1):1002--1017, 2019.

\bibitem[ES21]{EscherSearle21}
Christine Escher and Catherine Searle.
\newblock Torus actions, maximality, and non-negative curvature.
\newblock {\em J. Reine Angew. Math.}, 2021(780):221--264, 2021.

\bibitem[FG16]{FangGrove16}
Fuquan Fang and Karsten Grove.
\newblock Reflection groups in non-negative curvature.
\newblock {\em J. Differential Geom.}, 102(2):179--205, 16.

\bibitem[FGT17]{FangGroveThorbergsson17}
Fuquan Fang, Karsten Grove, and Gudlaugur Thorbergsson.
\newblock Tits geometry and positive curvature.
\newblock {\em Acta Mathematica}, 218(1):1--53, 2017.

\bibitem[FR04]{FangRong04}
Fuquan Fang and Xiaochun Rong.
\newblock Positively curved manifolds with maximal discrete symmetry rank.
\newblock {\em Amer. J. Math.}, 126(2):227--245, 2004.

\bibitem[FR05]{FangRong05}
Fuquan Fang and Xiaochun Rong.
\newblock Homeomorphism classification of positively curved manifolds with
  almost maximal symmetry rank.
\newblock {\em Math. Ann.}, 332(1):81--101, 2005.

\bibitem[FRW13]{FrankRongWang13}
Philipp Frank, Xiaochun Rong, and Yusheng Wang.
\newblock Fundamental groups of positively curved manifolds with symmetry.
\newblock {\em Math. Ann.}, 355(4):1425--1441, 2013.

\bibitem[{Gal}14]{GalazGarcia14}
Fernando {Galaz-Garc\'ia}.
\newblock {\em {\rm A Note on Maximal Symmetry Rank, Quasipositive Curvature,
  and Low Dimensional Manifolds. In:} Geometry of Manifolds with Non-negative
  Sectional Curvature}, volume 2110 of {\em Lecture Notes in Mathematics},
  pages 45--55.
\newblock Springer International Publishing, 2014.

\bibitem[GH82]{GroveHalperin82}
Karsten Grove and Stephen Halperin.
\newblock Contributions of rational homotopy theory to global problems in
  geometry.
\newblock {\em Publ. Math. Inst. Hautes Etudes Sci.}, 56(1):171--177, 1982.

\bibitem[GK14]{Galaz-GarciaKerin14}
Fernando {Galaz-Garc\'ia} and Martin Kerin.
\newblock Cohomogeneity-two torus actions on non-negatively curved manifolds of
  low dimension.
\newblock {\em Math. Z.}, 276(1-2):133--152, 2014.

\bibitem[GKR20]{GGKR20}
Fernando {Galaz-Garc\'ia}, Martin Kerin, and Marco Radeschi.
\newblock Torus actions on rationally elliptic manifolds.
\newblock {\em Math. Z.}, 2020.

\bibitem[GKS20]{GoetteKerinShankar20}
Sebastian Goette, Martin Kerin, and Krishnan Shankar.
\newblock Highly connected 7-manifolds and non-negative sectional curvature.
\newblock {\em Ann. of Math. (2)}, 191(3):829, 2020.

\bibitem[GKS21]{GoetteKerinShankar21}
Sebastian Goette, Martin Kerin, and Krishnan Shankar.
\newblock Fake lens spaces and non-negative sectional curvature.
\newblock In {\em Differential geometry in the large}, volume 463 of {\em
  London Math. Soc. Lecture Note Ser.}, pages 285--290. Cambridge Univ. Press,
  Cambridge, 2021.

\bibitem[GKW]{GorodskiKollrossWilking21preprint}
Claudio Gorodski, Andreas Kollross, and Burkhard Wilking.
\newblock Actions on positively curved manifolds and boundary in the orbit
  space.
\newblock preprint, arXiv:2112.00513.

\bibitem[Gro02]{Grove02}
Karsten Grove.
\newblock Geometry of, and via, symmetries.
\newblock In {\em Conformal, Riemannian and Lagrangian Geometry}, pages 31--53.
  American Mathematical Society, 2002.

\bibitem[GS94]{GroveSearle94}
Karsten Grove and Catherine Searle.
\newblock Positively curved manifolds with maximal symmetry-rank.
\newblock {\em J. Pure Appl. Algebra}, 91:137--142, 1994.

\bibitem[GS97]{GroveSearle97}
K.~Grove and C.~Searle.
\newblock {Differential topological restrictions curvature and symmetry}.
\newblock {\em J. Differential Geom.}, 47(3):530--559, 1997.

\bibitem[GS11]{GalazGarciaSearle11}
Fernando {Galaz-Garc\'ia} and Catherine Searle.
\newblock Low-dimensional manifolds with non-negative curvature and maximal
  symmetry rank.
\newblock {\em Proc. Amer. Math. Soc.}, 139(7):2559--2564, 2011.

\bibitem[GS14]{GalazGarciaSearle14}
Fernando {Galaz-Garc\'ia} and Catherine Searle.
\newblock Nonnegatively curved 5-manifolds with almost maximal symmetry rank.
\newblock {\em Geom. Topol.}, 18(3):1397--1435, 2014.

\bibitem[GVZ11]{GroveVerdianiZiller11}
Karsten Grove, Luigi Verdiani, and Wolfgang Ziller.
\newblock An exotic ${T}_1\mathbb{S}^4$ with positive curvature.
\newblock {\em Geom. Funct. Anal.}, 21(3):499--524, 2011.

\bibitem[GW14]{GroveWilking14}
Karsten Grove and Burkhard Wilking.
\newblock A knot characterization and 1-connected nonnegatively curved
  4-manifolds with circle symmetry.
\newblock {\em Geom. Topol.}, 18(5):3091--3110, 2014.

\bibitem[GW18]{GuijarroWilhelm18}
Luis Guijarro and Frederick Wilhelm.
\newblock Focal radius, rigidity, and lower curvature bounds.
\newblock {\em Proc. Lond. Math. Soc.}, 116(6):1519--1552, 2018.

\bibitem[GW20]{GuijarroWilhelm20}
Luis Guijarro and Frederick Wilhelm.
\newblock Restrictions on submanifolds via focal radius bounds.
\newblock {\em Math. Res. Lett.}, 27(1):115--139, 2020.

\bibitem[GW22]{GuijarroWilhelm22}
Luis Guijarro and Frederick Wilhelm.
\newblock A softer connectivity principle.
\newblock {\em Comm. Anal. Geom.}, 30(5):1093--1119, 2022.

\bibitem[GWY19]{GroveWilkingYeager19}
Karsten Grove, Burkhard Wilking, and Joseph Yeager.
\newblock Almost non-negative curvature and rational ellipticity in
  cohomogeneity two.
\newblock {\em Ann. Inst. Fourier (Grenoble)}, 69(7):2921--2939, 2019.

\bibitem[GWZ08]{GroveWilkingZiller08}
Karsten Grove, Burkhard Wilking, and Wolfgang Ziller.
\newblock Positively curved cohomogeneity one manifolds and {3-Sasakian}
  geometry.
\newblock {\em J. Differential Geom.}, 78(1), 2008.

\bibitem[Ham06]{Hambleton06}
Ian Hambleton.
\newblock Some examples of free actions on products of spheres.
\newblock {\em Topology}, 45(4):735--749, 2006.

\bibitem[HK89]{HsiangKleiner89}
Wu-Yi Hsiang and Bruce Kleiner.
\newblock On the topology of positively curved 4-manifolds with symmetry.
\newblock {\em J. Differential Geom.}, 29(3):615--621, 1989.

\bibitem[KSS]{KennardSamaniSearle21preprint}
Lee Kennard, Elahe~Khalili Samani, and Catherine Searle.
\newblock Positive curvature and discrete abelian symmetry.
\newblock preprint, arXiv:2110.13345.

\bibitem[KW17]{KennardWylie17}
Lee Kennard and William Wylie.
\newblock Positive weighted sectional curvature.
\newblock {\em Indiana Univ. Math. J.}, 66(2):419--462, 2017.

\bibitem[KWWa]{KennardWiemelerWilking22preprint}
Lee Kennard, Michael Wiemeler, and Burkhard Wilking.
\newblock Positive curvature, torus symmetry, and matroids.
\newblock {\em preprint}, arXiv:2212.08152.

\bibitem[KWWb]{KennardWiemelerWilking21preprint}
Lee Kennard, Michael Wiemeler, and Burkhard Wilking.
\newblock Splitting of torus representations and applications in the {Grove}
  symmetry program.
\newblock {\em preprint}, arXiv:2106.14723.

\bibitem[Mou]{MouilleWebpage}
Lawrence Mouill\'e.
\newblock Intermediate {Ricci} curvature,
  \url{https://sites.google.com/site/lgmouille/research/intermediate-ricci-curvature}.

\bibitem[Mou22a]{Mouille22a}
Lawrence Mouill\'e.
\newblock Local symmetry rank bound for positive intermediate {Ricci}
  curvatures.
\newblock {\em Geom. Dedicata}, 216(23), 2022.

\bibitem[Mou22b]{Mouille22b}
Lawrence Mouill\'e.
\newblock Torus actions on manifolds with positive intermediate {Ricci}
  curvature.
\newblock {\em J. Lond. Math. Soc.}, 106(4):3792--3821, 2022.

\bibitem[MY67]{MontgomeryYang67}
D.~Montgomery and C.~T. Yang.
\newblock Differentiable transformation groups on homotopy spheres.
\newblock {\em Michigan Math. J.}, 14(1), 1967.

\bibitem[OR70]{OrlikRaymond70}
Peter Orlik and Frank Raymond.
\newblock Actions of the torus on 4-manifolds. i.
\newblock {\em Trans. Amer. Math. Soc.}, 152(2):531--559, 1970.

\bibitem[Pera]{Perelman02preprint}
Grisha Perelman.
\newblock The entropy formula for the {Ricci} flow and its geometric
  applications.
\newblock preprint, arXiv:math/0211159.

\bibitem[Perb]{Perelman03preprint1}
Grisha Perelman.
\newblock Finite extinction time for the solutions to the {Ricci} flow on
  certain three-manifolds.
\newblock preprint, arXiv:math/0307245.

\bibitem[Perc]{Perelman03preprint2}
Grisha Perelman.
\newblock Ricci flow with surgery on three-manifolds.
\newblock preprint, arXiv:math/0303109.

\bibitem[Pet16]{Petersen16}
Peter Petersen.
\newblock {\em Riemannian Geometry}, volume 171 of {\em Graduate Texts in
  Mathematics}.
\newblock Springer International Publishing, $3^\mathrm{rd}$ edition, 2016.

\bibitem[PW]{PetersenWilhelm08preprint}
Peter Petersen and Frederick Wilhelm.
\newblock An exotic sphere with positive sectional curvature.
\newblock preprint, arXiv:0805.0812.

\bibitem[Rei]{Reiser22preprint}
Philipp Reiser.
\newblock Metrics of positive {Ricci} curvature on simply-connected manifolds
  of dimension $6k$.
\newblock preprint, arXiv:2210.15610.

\bibitem[Rei23]{Reiser23}
Philipp Reiser.
\newblock Generalized surgery on {R}iemannian manifolds of positive {R}icci
  curvature.
\newblock {\em Trans. Amer. Math. Soc.}, 376(5):3397--3418, 2023.

\bibitem[Ron02]{Rong02}
Xiaochun Rong.
\newblock Positively curved manifolds with almost maximal symmetry rank.
\newblock {\em Geom. Dedicata}, 95(1):157--182, 2002.

\bibitem[RWa]{ReiserWraith23preprint1}
Philipp Reiser and David Wraith.
\newblock A generalization of the {Perelman} gluing theorem and applications.
\newblock {\em preprint}, arXiv:2308.06996.

\bibitem[RWb]{ReiserWraith23preprint2}
Philipp Reiser and David Wraith.
\newblock Positive intermediate {Ricci} curvature on connected sums.
\newblock {\em preprint}, arXiv:2310.02746.

\bibitem[RWc]{ReiserWraith22preprint2}
Philipp Reiser and David Wraith.
\newblock Positive intermediate {Ricci} curvature on fibre bundles.
\newblock {\em preprint}, arXiv:2211.14610.

\bibitem[RW23]{ReiserWraith23}
Philipp Reiser and David~J. Wraith.
\newblock Intermediate {R}icci curvatures and {G}romov's {B}etti number bound.
\newblock {\em J. Geom. Anal.}, 33(12):Paper No. 364, 20, 2023.

\bibitem[Sha98]{Shankar98}
Krishnan Shankar.
\newblock On the fundamental groups of positively curved manifolds.
\newblock {\em J. Differential Geom.}, 49(1), 1998.

\bibitem[She90]{Shen90}
Zhongmin Shen.
\newblock {\em Finite topological type and vanishing theorems for {Riemannian}
  manifolds}.
\newblock PhD thesis, State University of New York at Stony Brook, 1990.

\bibitem[She93]{Shen93}
Zhongmin Shen.
\newblock On complete manifolds of nonnegative kth-{Ricci} curvature.
\newblock {\em Trans. Amer. Math. Soc.}, 338(1):289--310, 1993.

\bibitem[Smi38]{Smith38}
P.A. Smith.
\newblock {Transformations of a finite period}.
\newblock {\em Ann. of Math.}, 39:127--164, 1938.

\bibitem[Smi44]{Smith44}
P.A. Smith.
\newblock {Permutable periodic transformations}.
\newblock {\em Proc. Nat. Acad. Sci. U.S.A.}, 30:105--108, 1944.

\bibitem[Su63]{Su63}
J.~C. Su.
\newblock Transformation groups on cohomology projective spaces.
\newblock {\em Trans. Amer. Math. Soc.}, 106(2):305--318, 1963.

\bibitem[Sug82]{Sugahara82}
K.~Sugahara.
\newblock The isometry group and the diameter of a {Riemannian} manifold with
  positive curvature.
\newblock {\em Math. Japon.}, 27:631--634, 1982.

\bibitem[SY91]{ShaYang91}
Ji-Ping Sha and DaGang Yang.
\newblock Positive {Ricci} curvature on the connected sums of ${S}^n\times
  {S}^m$.
\newblock {\em J. Differential Geom.}, 33(1):127--137, 1991.

\bibitem[VZ18]{VerdianiZiller18}
Luigi Verdiani and Wolfgang Ziller.
\newblock Seven dimensional cohomogeneity one manifolds with nonnegative
  curvature.
\newblock {\em Math. Ann.}, 371(1-2):655--662, 2018.

\bibitem[Wil]{Wilking08preprint}
Burkhard Wilking.
\newblock Torus actions on manifolds with quasi-positive curvature.
\newblock unpublished manuscript, prepared by M. Kerin.

\bibitem[Wil97]{Wilhelm97}
Frederick Wilhelm.
\newblock On intermediate {Ricci} curvature and fundamental groups.
\newblock {\em Illinois J. Math.}, 41(3):488--494, 1997.

\bibitem[Wil03]{Wilking03}
Burkhard Wilking.
\newblock Torus actions on manifolds of positive sectional curvature.
\newblock {\em Acta Math.}, 191(2):259--297, 2003.

\bibitem[Wil06]{Wilking06}
Burkhard Wilking.
\newblock Positively curved manifolds with symmetry.
\newblock {\em Ann. of Math.}, 163(2):607--668, 2006.

\bibitem[Wil07]{Wilking07}
Burkhard Wilking.
\newblock A duality theorem for {Riemannian} foliations in nonnegative
  sectional curvature.
\newblock {\em Geom. Funct. Anal.}, 17(4):1297--1320, 2007.

\bibitem[Wol11]{Wolfson11}
Jon Wolfson.
\newblock Manifolds with $k$-positive {Ricci} curvature.
\newblock In {\em Variational Problems in Differential Geometry}, pages
  182--201. Cambridge University Press, 2011.

\bibitem[Wra07]{Wraith07}
David~J. Wraith.
\newblock New connected sums with positive {Ricci} curvature.
\newblock {\em Ann. Global Anal. Geom.}, 32(4):343--360, 2007.

\bibitem[WW22]{WalshWraith22}
Mark Walsh and David~J. Wraith.
\newblock H-space and loop space structures for intermediate curvatures.
\newblock {\em Commun. Contemp. Math.}, 2022.

\bibitem[Xia97]{Xia97}
Changyu Xia.
\newblock A generalization of the classical sphere theorem.
\newblock {\em Proc. Amer. Math. Soc.}, 125(1):255--258, 1997.

\bibitem[Zil06]{Ziller06}
Wolfgang Ziller.
\newblock Examples of manifolds with non-negative sectional curvature.
\newblock {\em Surv. Differ. Geom.}, 11(1):63--102, 2006.

\end{thebibliography}

\end{document}